\documentclass[11pt]{article} 
\usepackage[a4paper,margin=1in]{geometry} 
\usepackage{amsmath,amssymb,amsthm} 
\usepackage{amsfonts} 
\usepackage{float}
\usepackage{graphicx} 
\usepackage{booktabs,longtable}
\usepackage{hyperref} 
\usepackage{mathtools} 
\usepackage{paracol} 
\usepackage{pdflscape} 
\usepackage{caption} 
\usepackage{subcaption} % <--- ADDED THIS PACKAGE
\usepackage{xcolor} 
\usepackage{mathabx} 
\usepackage{tikz-cd}
\usepackage{placeins}
\usepackage{authblk}

\theoremstyle{plain}
\newtheorem{theorem}{Theorem}[section]

\newtheorem{lemma}[theorem]{Lemma}
\newtheorem{proposition}[theorem]{Proposition}
\newtheorem{corollary}[theorem]{Corollary}

\theoremstyle{definition}
\newtheorem{definition}[theorem]{Definition}

\theoremstyle{remark}
\newtheorem{remark}[theorem]{Remark}

\theoremstyle{remark}
\newtheorem{assumption}[theorem]{Assumption}
\theoremstyle{remark}

\DeclareMathOperator{\vol}{vol}

\DeclareMathOperator{\Spec}{Spec}
\newcommand{\sslash}{\mathbin{/\mkern-6mu/}}
\let\pr\relax
\DeclareMathOperator{\pr}{pr}
\let\shape\relax
\DeclareMathOperator{\shape}{shape}

\newcommand{\N}{\mathbb{N}}
\newcommand{\Prob}{\mathbb{P}}

\newcommand{\C}{\mathbb{C}}
\newcommand{\R}{\mathbb{R}}

\title{Metric Rarity and the Emergence of Symmetry in G-Invariant Potential Surfaces}
\author{Irmi Schneider}
\affil{School of Computer Science and Engineering, The Hebrew University of Jerusalem}
\affil{\textit{irmi.schneider@mail.huji.ac.il}} 
\date{January 27, 2026}
\begin{document} 
\maketitle 

\begin{abstract}
Let $X$ be an irreducible complex affine algebraic variety defined over $\mathbb{R}$, equipped with a faithful action of a finite group $G$, and let $Y = X \sslash G$ denote the categorical quotient with projection $\pi$. We study the geometry of the real image $L = \pi(X(\mathbb{R})) \subset Y(\mathbb{R})$ and its consequences for $G$-invariant optimization.

Equipping $Y(\mathbb{R})$ with the measure induced by a $G$-invariant metric on $X$, we prove that the relative volume of $L$ in $Y(\mathbb{R})$ equals $(\#\mathrm{Inv}(G))^{-1}$, where $\mathrm{Inv}(G)$ is the set of involutions of $G$. For the symmetric group $S_n$ acting on $\mathbb{R}^n$, this ratio decays super-exponentially in $n$. In particular, $L$ is metrically rare within the ambient real quotient.

We apply this result to two phenomena observed in $G$-invariant optimization problems:

Regime~I (Rarity of asymmetric critical points). The super-exponential decay of the volume of $L$ renders the interior $L^\circ$ statistically negligible as a locus for critical points. This geometric rarity provides a rationale for the observed prevalence of symmetry: generic critical points are constrained to the boundary strata of $L$, corresponding to orbits with non-trivial stabilizers.

Regime~II (Energetic ordering by symmetry). We formulate the \emph{Active Constraint hypothesis}: due to the metric rarity of the real image $L$, the landscape is dominated by a global gradient that drives the deepest descent trajectories toward the boundary of $L$. This global gradient directs the global minimum into the high-codimension strata of the boundary---corresponding to large stabilizers---thereby establishing a structural link between low energy and non-trivial stabilizers. This mechanism rationalizes the funnel topography of Lennard-Jones clusters, where the system is funneled into a crystallized ground state.
\end{abstract}

\tableofcontents
\clearpage

\section{Introduction}
\subsection{Problem Statement: The Prevalence of Symmetry}

\section{Introduction}
\subsection{Problem Statement: The Prevalence of Symmetry}

Let $X$ be an irreducible complex affine algebraic variety defined over $\mathbb{R}$, 
equipped with a faithful action of a finite group $G$ by real automorphisms. We denote by $X(\mathbb{R})$ the set of real points of $X$; we refer to points $x \in X$ as \textbf{configurations}.
For any configuration $x \in X$, let $G_x = \{g \in G \mid g \cdot x = x\}$ denote its stabilizer subgroup. We classify a configuration $x$ as \textbf{symmetric} if its stabilizer $G_x$ is non-trivial, and \textbf{asymmetric} (or in general position) otherwise.

Let $f: X \to \mathbb{C}$ be a $G$-invariant regular function defined over $\mathbb{R}$ (i.e., $f \in \mathbb{R}[X]^G$).
Throughout this paper, we will denote the restriction of $f$ to the real locus $X(\mathbb{R})$ simply by $f$, distinguishing between the complex and real functions only when necessary for clarity.
Our primary focus is the behavior of $f$ on the real locus $X(\mathbb{R})$. 
Within this setting, a central question regarding the critical points is the emergence of symmetry: 
we investigate the extent to which these points possess non-trivial stabilizers.

The existence of symmetric critical points is, in itself, a standard consequence of the group action. 
By the Principle of Symmetric Criticality (see Theorem \ref{thm:PSC}), any critical point of $f$ restricted to a real fixed-point subspace (definition \ref{def:fixed_point_space}) 
is necessarily a critical point of $f$ on the entire real manifold; thus, the existence of symmetric critical points is almost guaranteed.

Consider the geometry of the real locus $X(\mathbb{R})$. Within this manifold, the set of symmetric configurations forms a lower-dimensional 
subvariety of measure zero. Conversely, the asymmetric configurations (those in general position) constitute an open, dense set, occupying the full 
measure of the configuration space. Consider a $G$-invariant energy landscape 
$f: X(\mathbb{R}) \to \mathbb{R}$ with a large number of critical points. 
Statistically, one would expect the majority of these points to lie in the generic, asymmetric stratum.
Moreover, this statistical dominance implies a specific prediction for the ground state: 
assuming the energy values are distributed randomly across the critical points, the probability that the global minimum coincides with a rare 
symmetric configuration should be negligible. 

Empirical observation, however, reveals two distinct deviations from this expectation.

First, in what we term \textbf{Regime I}, symmetry appears to be prevalent rather than rare. 
This phenomenon was first empirically observed in the context of shallow ReLU networks by Safran and Shamir~\cite{safran2018}, who noted that all the local minima they detected consistently exhibited a highly symmetric structure. Building on this observation, Arjevani et al.~\cite{arjevani2019, arjevani2021} rigorously documented that these optimization landscapes systematically preserve the permutation symmetries of hidden neurons---exhibiting maximal or large isotropy subgroups---and that large isotropy groups similarly persist in the landscape of symmetric tensor decomposition.

In companion numerical experiments~\cite{schneider2025}, we investigated this question across a broad class of $G$-invariant variational problems of the following common form:
a finite group $G$ acts faithfully and linearly on a real vector space $X(\mathbb{R})$, and $f:X(\mathbb{R})\to\mathbb{R}$ is a $G$-invariant polynomial.
This framework simultaneously covers, for example, particle interaction models (with $G=S_n$ acting by label permutations), graph-based objective functions with automorphism symmetries, and invariant functions on finite projective geometries (with $G$ a finite linear group such as $PGL_n(\mathbb{F}_q)$ acting on the relevant representation space).
Across these algebraically distinct settings— and especially at higher polynomial degrees —we observed that almost all numerically detected critical points had non-trivial stabilizers, contrary to the statistical expectation.

Second, a distinct preference for \emph{higher} symmetry emerges at the energetic extremes—a phenomenon we term \textbf{Regime II}. Even in the cases mentioned above where symmetry is the norm, the deepest critical points exhibit the largest stabilizers (Figure~\ref{fig:distinct_values}).
Additional numerical evidence and robustness checks appear in Appendix~\ref{app:boundary_experiments}.
% Bibliography entries:
% \bibitem{safran2018}
% I.~Safran and O.~Shamir,
% \textit{Spurious Local Minima are Common in Two-Layer ReLU Neural Networks}.
% \href{https://arxiv.org/abs/1712.08968}{arXiv:1712.08968} \;[\href{https://arxiv.org/pdf/1712.08968.pdf}{PDF}], 2018.
%
% \bibitem{arjevani2019}
% Y.~Arjevani and M.~Field,
% \textit{On the principle of least symmetry breaking in shallow ReLU models}.
% \href{https://arxiv.org/abs/1912.11939}{arXiv:1912.11939} \;[\href{https://arxiv.org/pdf/1912.11939.pdf}{PDF}], 2019.
% Add this to your bibliography section:
% \bibitem{safran2018}
% I.~Safran and O.~Shamir,
% \textit{Spurious Local Minima are Common in Two-Layer ReLU Neural Networks}.
% \href{https://arxiv.org/abs/1712.08968}{arXiv:1712.08968} \;[\href{https://arxiv.org/pdf/1712.08968.pdf}{PDF}], 2018.
\begin{figure}[t!]
\centering
\begin{subfigure}[t]{0.48\textwidth}
        \centering
        \includegraphics[width=\linewidth]{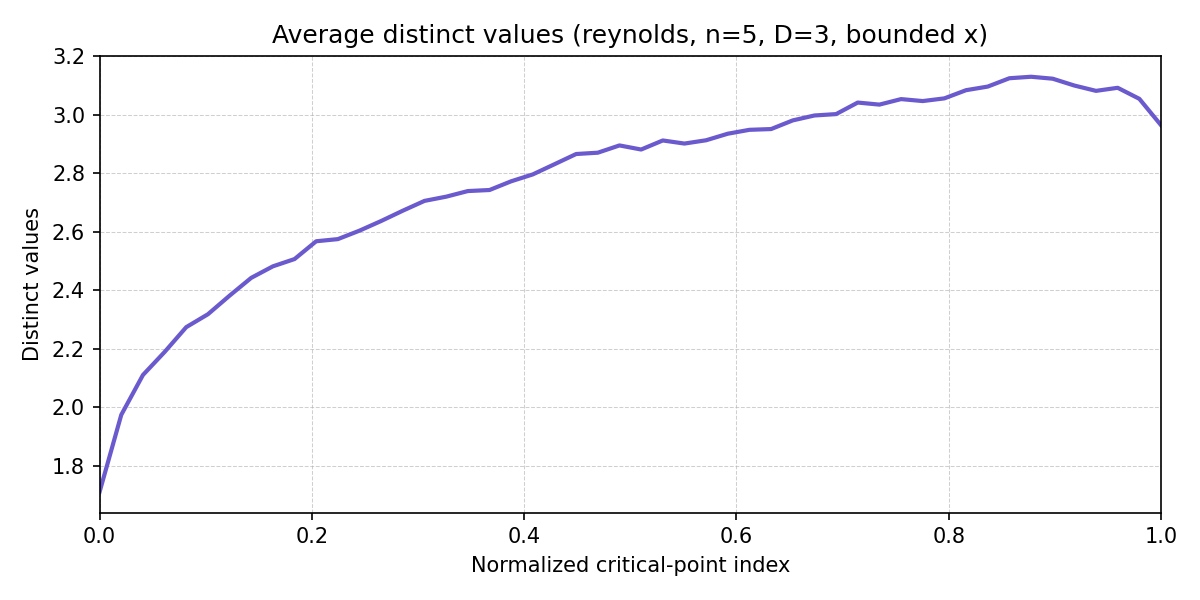}
        \caption{\textbf{Symmetry amplification at the global minimum.} \\
        \textbf{Model \& Data:} We analyze 1,000 random $S_5$-invariant polynomials generated via Reynolds symmetrization (Definition~\ref{def:reynolds_symmetrization}) of degree-3 polynomials sampled from the \emph{non-homogeneous} KSS ensemble (Definition~\ref{def:nhkss}) (with a confining term $+\|x\|^4$). For each polynomial, we performed 1,000 Newton searches to exhaustively detect critical points (filtered to be unique up to permutation). \\
        \textbf{Metric:} Critical points are sorted by energy from global minimum ($t=0$) to highest saddle ($t=1$). The curve displays the \textbf{ensemble mean} of the asymmetry score $D(x)$ (number of distinct coordinate values). \\
        \textbf{Observation:} A structural ordering emerges. While the bulk of the spectrum ($t \approx 0.5$) is dominated by generic asymmetric configurations (avg $D \approx 2.88$), the global minima ($t=0$) are statistically driven to high-symmetry strata (avg $D \approx 1.48$), confirming that low energy correlates with large stabilizer.}
        \label{fig:distinct_values}
    \end{subfigure}
    \hfill 
    \begin{subfigure}[t]{0.48\textwidth}
      \centering
      \includegraphics[width=\linewidth]{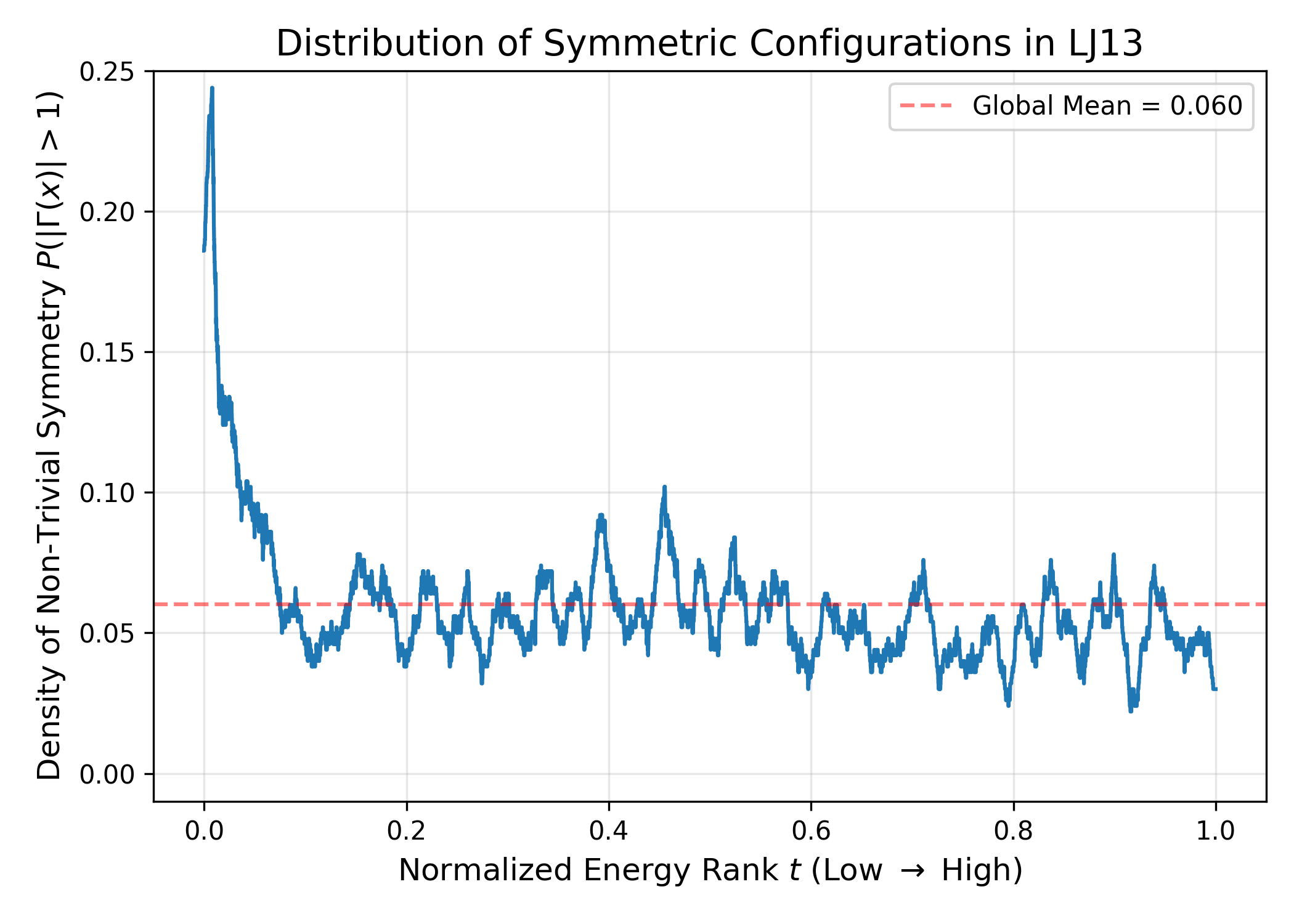}
      \caption{\textbf{The localization of symmetry in the Lennard-Jones (LJ13) landscape.} \\
      \textbf{Data:} Analysis of 30,517 critical points of the $LJ_{13}$ potential, sorted by energy from the global minimum ($t=0$) to the highest index saddle ($t=1$).\\
      \textbf{Metric:} The y-axis represents the local density of non-trivial symmetries, estimated via a moving average (window $w=500$) of the indicator function $\mathbb{I}(|\Gamma(x)| > 1)$. We measure the binary property of having \emph{any} symmetry versus none.\\
      \textbf{Observation:} A clear separation between two regimes is visible. The deep energy funnel ($t \to 0$) exhibits a high concentration of symmetric configurations (up to 20\% density). As energy increases, the probability of encountering symmetry decays rapidly towards the global baseline (dashed red line), confirming that symmetry is structurally linked to low-energy order rather than being uniformly distributed.}
      \label{fig:lj13_symmetry}
    \end{subfigure}

    \caption{Regime II symmetry emergence: (a) Random $S_n$ invariant polynomials. (b) Physical potential (LJ13).}
  \end{figure}

This correlation becomes the dominant selection mechanism in complex physical systems such as Lennard-Jones (LJ) clusters, 
governed by the potential $V_{LJ}$ (Definition~\ref{def:general_potential}). 
These systems map naturally to our algebraic framework: the configuration space \textit{modulo global motions} 
acts as the variety $X_{\shape}$ (Definition~\ref{def:shape_space}), and the physical \textit{point group} 
is identified with the stabilizer in $S_n$ (Proposition~\ref{prop:point_group_iso}).

In this setting, the generic stratum has full measure: for 13 particles (LJ13), approximately 94\% of the 30,517 known critical points and approximately 78\% of the local minima documented in the Cambridge Cluster Database\footnote{\url{https://www-wales.ch.cam.ac.uk/CCD.html}} possess a trivial stabilizer. 
Thus, the first phenomenon (prevalence) is absent.
However, the drive toward symmetry at the extremes re-emerges. 
As illustrated in Figure~\ref{fig:lj13_symmetry}, the local density of symmetric configurations increases monotonically as one moves deeper into the potential wells, culminating in the global minimum: the regular icosahedron (order 120).
This prevalence of symmetry is not unique to $N=13$; the pattern persists across the range $LJ_1$--$LJ_{150}$, where approximately 83\% of global minima possess non-trivial point-group symmetry \cite{wales1997global}.

The systematic correlation between high symmetry and energetic extremes---specifically the tendency of symmetric structures to occupy the tails of the energy distribution---was originally documented by Wales~\cite{wales1998symmetry}. 
Through extensive surveys of model systems, Wales established that while symmetry might not dictate the average energy, it is a defining characteristic of the lowest and highest energy states. 
We note that Wales analyzed \emph{near-symmetry} (defining symmetry via continuous measures and tolerance thresholds), whereas our framework focuses on \emph{exact} symmetry. 
Notably, the selection of symmetry at energetic extremes remains robust even under this strict algebraic definition.

\subsection{Our Proposal: Geometry of the Real Image}
\label{subsec:proposal_geometry}

We shift our analysis from the configuration space $X$ to the quotient space $Y = X \sslash G$. Let $\pi: X \to Y$ 
denote the quotient map (see Section \ref{sec:algebraic_framework}). The physical configurations form the \textbf{real image} $L = \pi(X(\mathbb{R}))$, 
a semi-algebraic subset of the real quotient $Y(\mathbb{R})$. There is a canonical isomorphism between $G$-invariant functions 
on $X$ and regular functions on $Y$. Any invariant energy $f$ factors as $f = \tilde{f} \circ \pi$ for a 
unique $\tilde{f}: Y(\mathbb{R}) \to \mathbb{R}$.
In particular, this implies that invariance imposes \textbf{no structural constraint} on $\tilde{f}$ itself; $\tilde{f}$ can be any generic function.
Consequently, the landscape features of $f$ are determined solely by the interaction between the function $\tilde{f}$ and the 
geometry of the domain $L$. 

The critical points are governed by the chain rule:
\[
df_x = d\tilde{f}_{\pi(x)} \circ d\pi_x.
\]
Let $x$ be a smooth point of $X$. Proposition \ref{thm:surjectivity_diff} states that the differential $d\pi_x$ is surjective if and only if the stabilizer $G_x$ is trivial.
This implies a fundamental distinction:
\begin{enumerate}
    \item \textbf{Trivial Stabilizer (Smooth Region):} If $G_x$ is trivial, then $x$ is a critical point of $f$ if and only if $y=\pi(x)\in L^\circ$ is a critical point of $\tilde{f}$ (i.e., $d \tilde{f}_y = 0$).
    \item \textbf{Non-Trivial Stabilizer (Singular Region):} If $G_x$ is non-trivial, $d\pi_x$ is singular. Here, $x$ can be a critical point of $f$ even if $d \tilde{f}_y \neq 0$ (defined on the Zariski tangent space).
\end{enumerate}

Utilizing the stratification analysis of Procesi and Schwarz \cite{procesi1985}, we identify the topological boundary $\partial L$ with the locus of configurations possessing even-order stabilizers (Theorem \ref{thm:boundary_parity}).
In reflection groups such as $S_n$, every non-trivial stabilizer has even order; 
consequently, $\pi$ maps all symmetric points to the boundary $\partial L$. 
A similar restriction appears empirically for Lennard--Jones critical points, where odd-order symmetries are exceptionally rare (a phenomenon we discuss later).
Therefore, throughout this paper we adopt a convention of brevity: we frequently describe the properties of a configuration $x$ 
directly in terms of its image $\pi(x)$, referring to symmetric configurations simply as lying on the boundary $\partial L$.

We investigate the origin of this statistical concentration of critical points on the boundary $\partial L$. 
This tendency manifests in varying degrees: from cases where the interior $L^\circ$ is devoid of critical points (Regime I), to cases where the boundary simply captures the global minimum (Regime II).

\subsubsection{The Metric Rarity of the Real Image}
\label{subsec:metric_rarity}

The dichotomy between the ``Smooth Region'' (trivial stabilizer, interior $L^\circ$) and the ``Singular Region'' (non-trivial stabilizer, boundary $\partial L$) derives its significance from the following geometric fact: the real image $L$ is metrically rare within the quotient space $Y(\mathbb{R})$. While $L$ and $Y(\mathbb{R})$ share the same topological dimension, the volume of $L = \pi(X(\mathbb{R}))$ becomes increasingly small as the system size grows. This rarity is universal across our models.

Consider first the case of symmetric polynomials, where the configuration space is $X(\mathbb{R})=\mathbb{R}^n$ and the group is the symmetric group $G=S_n$. 
Here, the quotient map $\pi:X\to Y$ sends $x=(x_1,...,x_n)$ to the elementary symmetric polynomials $y = (e_1, \ldots, e_n)$. These coordinates naturally parametrize the monic polynomial $P_y(t) = t^n - e_1 t^{n-1} + \dots + (-1)^n e_n$. Consequently, the real image $L$ represents the specific subset of parameters $y$ yielding monic polynomials with \emph{all real roots}. 
Already in the cubic case this region occupies only a small portion of coefficient space; Figure \ref{fig:s3_stratification} visualizes $L = \pi(\mathbb{R}^3)\subset \mathbb{R}^3$ for $n=3$.

We equip the quotient space $Y(\mathbb{R}) \cong \mathbb{R}^n$ with a Gaussian measure. Under the hypothesis that the Monic and Kac Gaussian ensembles share the same asymptotic decay rate (Assumption \ref{ass:ensemble_equivalence}), Theorem \ref{thm:rarity_kac} implies that the measure of the real image $L$ decays super-exponentially ($\sim e^{-C n^2}$).

\begin{figure}[htbp!]
    \centering
    \includegraphics[width=0.7\linewidth]{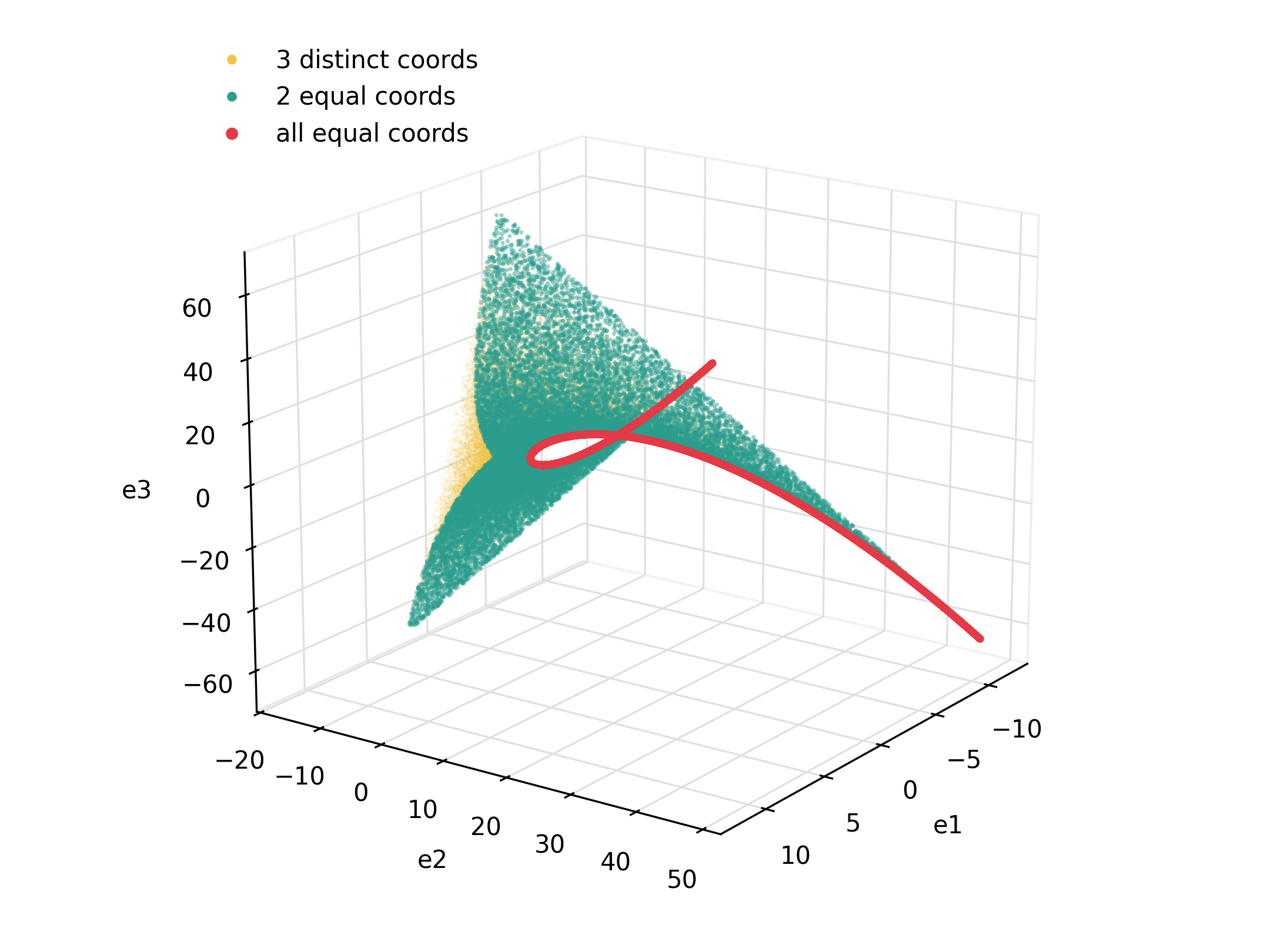} 
    \caption{\textbf{Visualization of the real image of $\pi: X(\mathbb{R}) \to Y(\mathbb{R})$ for $n=3$.} 
    The scatter plot displays the real image $L$ of the quotient map $\pi: \mathbb{R}^3 \to \mathbb{R}^3$ defined by elementary symmetric polynomials, given by $\pi(x_1,x_2,x_3) = (e_1, e_2, e_3) = (x_1 + x_2 + x_3, x_1 x_2 + x_2 x_3 + x_3 x_1, x_1 x_2 x_3)$.
    The geometry reveals a hierarchy of singularities corresponding to the stabilizer subgroups discussed in the text:
    \textbf{White (Ambient Space, $Y(\mathbb{R})$):} The full quotient space represents all real coefficients of monic cubic polynomials.
    \textbf{Yellow (Smooth Region, $L^\circ$):} Points with trivial stabilizers ($x_i \neq x_j$) form the 3D volume, representing coefficients of polynomials with three distinct real roots.
    \textbf{Green (Boundary, $\partial L$):} Points with $S_2$ stabilizers (two equal coordinates) form a 2D smooth boundary fold, representing coefficients of polynomials with 2 distinct real roots (one double root).
    \textbf{Red (Deep Singularity):} Points with full $S_3$ symmetry ($x_1=x_2=x_3$) collapse onto a 1D "cusp" or edge, representing coefficients of polynomials with a triple real root.}    \label{fig:s3_stratification}
\end{figure}

For a general finite group $G$ acting faithfully and linearly on a vector space $X(\mathbb{R})$, any $G$-invariant Hermitian metric on $X(\mathbb{C})$ induces a natural measure $\mu$ on $Y(\mathbb{R})$, as formally defined in Definition \ref{def:algebraic_involutions_measure}. 
In this metric setting, we measure the relative size of the real image within a ball $B(0,r)\subset Y(\mathbb{R})$ by considering the ratio of $\mu(L\cap B(0,r))$ to the total volume $\mu(B(0,r))$. 
Theorem \ref{thm:rarity_real_image_involutions} shows that this fraction is exactly $\frac{1}{\#\text{Inv}(G)}$ and does not depend on $r$, where $\text{Inv}(G)$ denotes the set of involutions in $G$.
This result provides an alternative measure for the real image in the case of $G = S_n$ permutation action on $\mathbb{R}^n$, where Corollary \ref{cor:Inv_Sn_decay} implies that the term $\frac{1}{|\text{Inv}(G)|}$ decays as $\sim e^{-C n\log(n)}$.

This rarity extends to the setting of interacting particles.
In this context, we formulate the system over the \textbf{physical shape space} $X_{\shape}^+ \subset X_{\shape}(\mathbb{R})$, defined by particle positions modulo global motions (see Definition \ref{def:shape_space}).
The geometric structure is established by assuming any Hermitian metric on the ambient configuration space of centered particle systems $\mathcal{V} = \mathbb{C}^{d \times (n-1)}$.
This metric descends canonically to the Root-Mean-Square Deviation (RMSD) metric on the quotient $X_{\shape}^{\mathrm{\pr}}$ (see Definition \ref{def:rmsd_distance}), which subsequently induces a natural measure $\mu$ on the quotient space $Y_{\shape}(\mathbb{R})$ (as constructed in Definition \ref{def:algebraic_involutions_measure_shape}).

With this metric framework in place, Theorem \ref{thm:relative_volume_real_image_shape} reveals that the domain of physically valid configurations $L=\pi(X_{\shape}^+)$ occupies a negligible fraction of the ambient quotient space.
Specifically, Corollary \ref{cor:Inv_Sn_decay_shape} establishes that the relative volume within a ball of radius $r$ decays exponentially:
\[
\frac{\mu(L\cap B(r,0))}{\mu(B(r,0))} = \frac{1}{2^{\min(d,n)} \cdot \#\mathrm{Inv}(S_n)} \sim e^{-C n\log(n)}.
\]

Consequently, for high-dimensional systems, $L$ is metrically rare within $Y_{\shape}(\mathbb{R})$. This geometric constraint forces the system into two distinct phenomenological regimes:

\subsubsection{Regime I: The ``Empty Interior'' (No Interior Critical Points)}
In polynomial optimization problems invariant under the action of a finite group, as well as in shallow neural networks invariant under permutations of hidden neurons (Arjevani et al.~\cite{arjevani2019, arjevani2021}, Schneider~\cite{schneider2025} and Figure~\ref{fig:distinct_values}), we observe prevalent symmetry. Geometrically, this corresponds to the scenario where the descended function $\tilde{f}$ possesses \textbf{no critical points} within the smooth interior of the real image $L$.
\emph{Equivalently}, the condition $d\tilde{f}=0$ has no solutions in $L^\circ$.
We propose that this arises from a volume disparity: since the physical domain $L$ occupies an increasingly small volume relative to the ambient quotient, the strict condition $d\tilde{f}=0$ is statistically unlikely to be satisfied within the interior $L^{\circ}$. Consequently, the critical points of $f$ arise exclusively from the singularities of the quotient map $\pi$---which correspond precisely to configurations with non-trivial stabilizers.

We illustrate this mechanism via a heuristic example. Consider an $S_n$-invariant polynomial $f:\mathbb{R}^n \to \mathbb{R}$ such that its quotient descendant $\tilde{f}$ is of degree $n$. By Bézout's theorem, such a function possesses at most $n^n$ critical points. Proceeding under Assumption \ref{ass:ensemble_equivalence}—specifically, that the real-root probability for the Monic Gaussian ensemble decays asymptotically as in the Kac ensemble—Corollary \ref{cor:rarity_monic} establishes that the measure of the physical domain $L$ decays super-exponentially. Consequently, the expected number of critical points falling within the interior $L^\circ$ vanishes, implying a statistical prohibition on trivial stabilizers in the asymptotic limit.

\paragraph{Recursive Rarity and High Symmetry}
The geometric rarity established above extends recursively across the lattice of subgroups.
Consider the search for critical points possessing a specific symmetry group $H \subseteq G$. By the Principle of Symmetric Criticality (Theorem \ref{thm:PSC}), this problem reduces to finding critical points of the restriction $f|_{X(\mathbb{R})^H}$ on the fixed-point subspace.

Within this stratum, the geometric landscape is defined by the action of the normalizer $N_G(H)$ on $X^H$, effectively factoring through the group $\overline{G} = N_G(H)/H$ (as defined in Definition \ref{def:normalizer_effective}).
Consequently, whether a generic critical point exists in this stratum depends on the "size" of the stratum's real image $L_H = \pi(X(\mathbb{R})^H)$ within the stratum quotient $Y^{(H)}(\mathbb{R})$.

Applying Lemma \ref{lem:stratum_volume}, we observe that the relative volume of this image is governed by the inverse count of involutions in the effective group:
\[
\frac{\operatorname{vol}\left(L_H\right)}{\operatorname{vol}\left(Y^{(H)}(\mathbb{R})\right)} = \frac{1}{\#\mathrm{Inv}(\overline{G})}.
\]
This relation uncovers a cascading selection mechanism. For small subgroups $H$, the effective group $\overline{G}$ remains large (comparable to $S_n$), resulting in a large number of involutions and a metrically rare real image $L_H$. In this regime, finding a critical point is statistically improbable. 
 However, as we restrict our attention to strata with larger stabilizers $H$, the effective group $\overline{G}$ shrinks, the number of involutions $\#\mathrm{Inv}(\overline{G})$ decreases, and the relative volume of the real image increases.
This monotonicity implies that the system is statistically driven away from low-symmetry strata and resides only where the symmetry $H$ is sufficiently high to render the real image $L_H$ voluminous enough to capture the critical points of the potential.

In Section \ref{sec:empirical_implications}, we formalize the connection between this theoretical mechanism and the empirical data. We translate the results of \cite{arjevani2019, arjevani2021, schneider2025} into this algebraic language, demonstrating how the ``Empty Interior'' phenomenon (no critical points of $\tilde{f}$ in $L^\circ$) provides an explanation for the experimentally observed prevalence of symmetry.

\subsubsection{Regime II: The Active Constraint (A Geometric Hypothesis)}

\label{subsubsec:active_constraint}

While the metric rarity established in Regime~I successfully explains the baseline \textit{prevalence} of symmetry, 
it leaves a fundamental phenomenon unaccounted for: the energetic \textit{ordering}. The metric argument implies that critical 
points should reside in the first stratum $L_H$ that is voluminous enough within $Y^{(H)}(\R)$ to statistically 
accommodate them. However, this fails to explain why, across both random $S_n$-invariant polynomials and physical 
clusters (e.g., $LJ_{13}$), the deepest minima exhibit \emph{higher} symmetry than the bulk critical points.

In $LJ_{13}$, the ``Empty Interior'' condition does not hold---approximately 94\% of critical points and 78\% of local minima are asymmetric. 
Nevertheless, the landscape is governed by a strict \textbf{energetic ordering}: the global minimum is attained at the unique, maximally symmetric icosahedron. 
Furthermore, this is not an outlier; the density of non-trivial symmetries increases monotonically as energy decreases, effectively localizing the deepest critical points within the high-symmetry strata of the boundary $\partial L$.
To address this, we propose a second geometric mechanism: the \textbf{Active Constraint}.

\paragraph{Heuristic description.}
We visualize the landscape of $\tilde{f}|_L$ through a geometric analogy: consider a bounded domain $L$ (analogous to a ``small carpet'') placed upon a vast, rugged terrain (analogous to the ``Himalayas''). If the domain were large, it might capture multiple deep valleys, potentially containing the global minimum in its interior. However, due to the metric rarity established above, $L$ is effectively small relative to the ambient landscape's variation.

Consequently, the domain is likely to rest upon a single dominant slope. While local non-convexities may create shallow interior minima, the overarching ambient gradient drives the global minimum toward the boundary $\partial L$, where the descent is geometrically arrested.
In our algebraic framework, the boundary $\partial L$ is the locus of non-trivial stabilizers (see Remark \ref{rem:symmetry_boundary_identification}), and strata of higher codimension correspond to larger stabilizers. Thus, the global minimum is driven to a high-codimension stratum of $\partial L$, where the global descent is intercepted by the geometry of the domain.

\begin{figure}[H]
  \centering
  \begin{subfigure}[b]{0.40\textwidth}
    \centering
    \includegraphics[width=\linewidth]{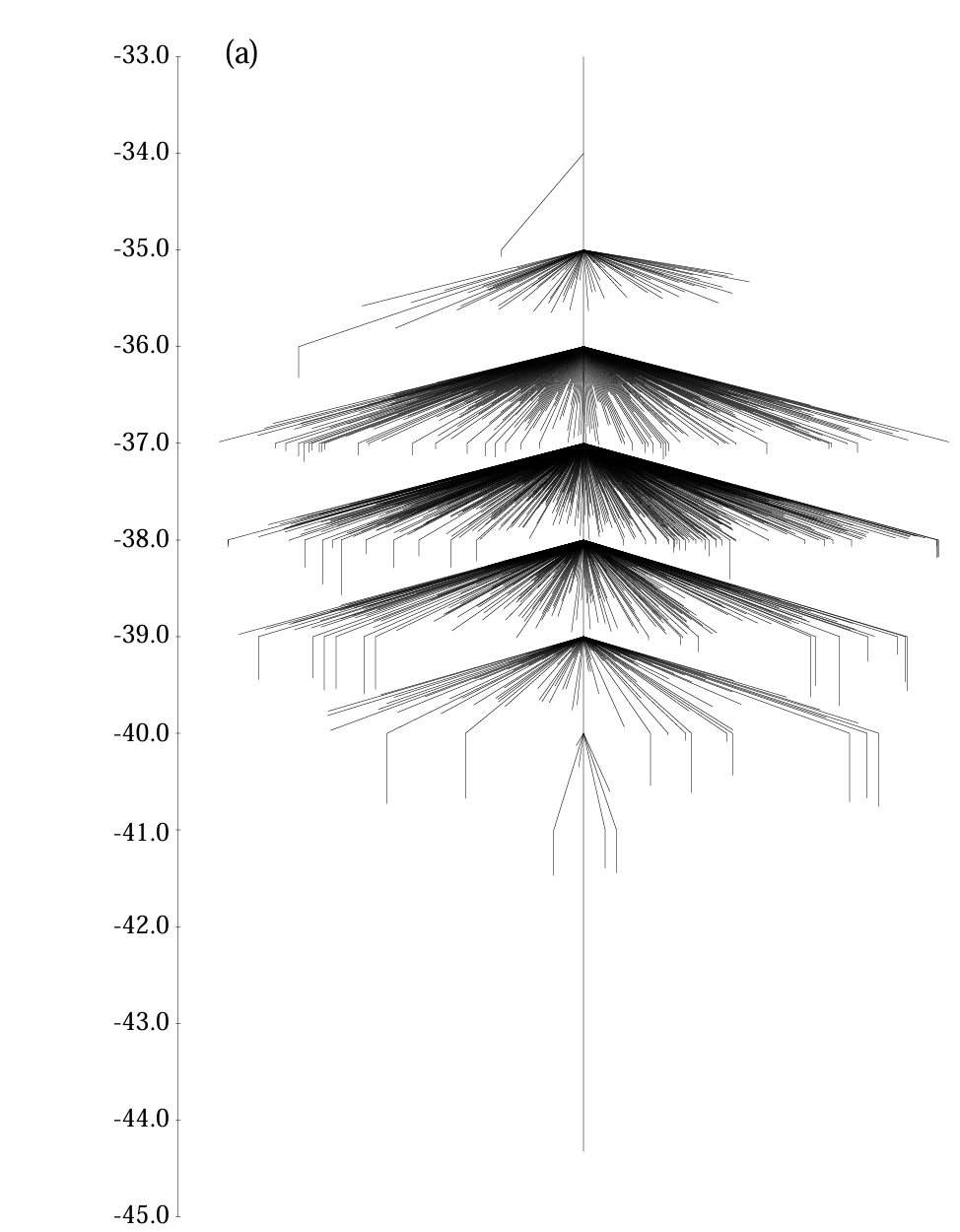}
    \caption{$LJ_{13}$}
    \label{fig:LJ13_disconnectivity}
  \end{subfigure}
  \hfill
  \begin{subfigure}[b]{0.40\textwidth}
    \centering
    \includegraphics[width=\linewidth]{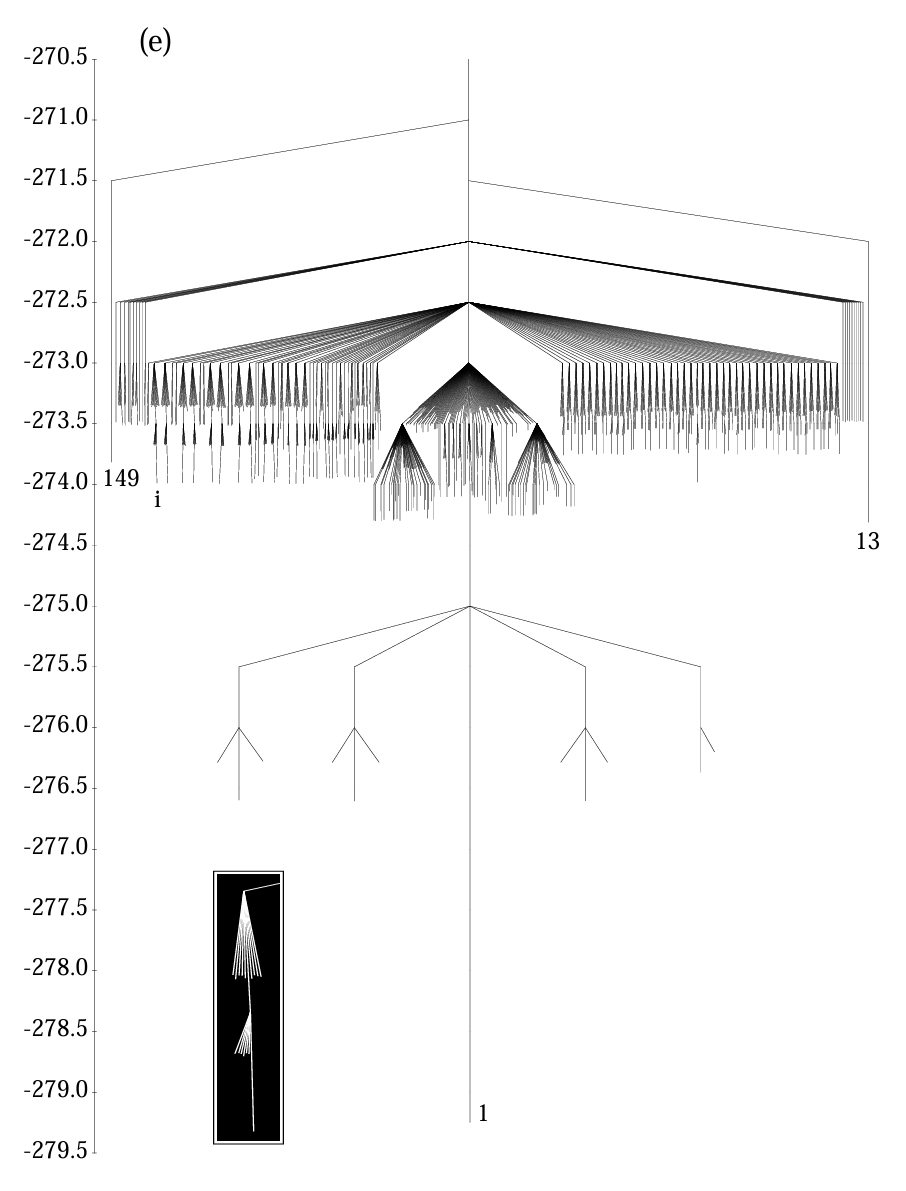}
    \caption{$LJ_{55}$}
    \label{fig:LJ55_disconnectivity}
  \end{subfigure}
\caption{\textbf{Topographical evolution of the real image $L$.} 
The disconnectivity graphs (visualized here as merge trees of energy sublevel sets) for $LJ_{13}$ (left) and $LJ_{55}$ (right) exhibit a unified "funnel" structure. 
All minima are shown for $LJ_{13}$, while for $LJ_{55}$ only branches leading to the 900 lowest-energy minima are displayed.
Numbers adjacent to nodes indicate how many minima the node represents, and selected branches are labeled by the energetic rank of their minima.
In the $LJ_{55}$ panel, the black inset is a zoom-in of the branch marked by $i$.
Reprinted with permission from J. P. K. Doye, M. A. Miller, and D. J. Wales, \textit{J. Chem. Phys.} \textbf{111}, 8417 (1999). Copyright 1999 AIP Publishing.}  \label{fig:disconnectivity_comparison}
\end{figure}

\paragraph{Physical Evidence: The Funnel Landscape.}

This intuition finds concrete support in Energy Landscape Theory \cite{wales1997global}. As visualized in Figure \ref{fig:disconnectivity_comparison}, Lennard-Jones clusters are characterized by ``funnel'' topographies---hierarchical organizations structured as a single, deep, and broad attraction basin containing relatively small sub-basins.

From our perspective, a "funnel" can be understood as the physical manifestation of a \textbf{coherent global gradient} of $\tilde{f}$ across the domain $L\subset Y(\R)$. 
Even when the landscape $\tilde{f}$ is multimodal, the essential geometric feature remains: the "global gradient" does not vanish in the interior. 
Instead, it creates a persistent downhill drift that directs the system outward, constraining the global minimum to the boundary $\partial L$. In this framework, the regular icosahedron of $LJ_{13}$ corresponds to a high-codimension stratum of the real image $L$---the locus where the global descent of the potential $\tilde{f}\mid_{Y(\R)}$ reaches the boundary of the domain $L$.

\paragraph{The $S_n$ invariant Polynomial Case.}
While the physical evidence strongly supports the Active Constraint mechanism via the funnel topology, 
the parallel behavior observed in random $S_n$-invariant polynomials implies that similar gradients might emerge in 
\textbf{generic random ensembles} restricted to a small domain $L$, absent of any specific physical structure.
However, establishing whether this mechanism drives the symmetry selection in these \textbf{unstructured models} 
remains an open challenge.

Future work is required to rigorously formulate the definition of a "global gradient," and subsequently to verify whether this 
gradient is indeed the causal factor driving the statistical preference for the boundary (and thus, high symmetry) in random 
models restricted to the real image $L$.

\subsection{Statement of Contributions: Results and Proposed Mechanisms}
We distinguish explicitly between rigorous geometric results, empirical findings, and the mechanisms proposed to explain them.

\paragraph{I. Mathematical Results.}
We consider the general setting where $X$ is an irreducible complex affine algebraic variety defined over $\mathbb{R}$, equipped with a faithful action of a group $G$. Within this framework, we establish the following theorems:

\begin{enumerate}
    \item \textbf{Probabilistic Rarity (The Gaussian $S_n$ Case) (Theorem \ref{thm:rarity_kac}):} 
    Specializing to $G=S_n$, we consider the standard Gaussian measure directly on the quotient space $Y(\mathbb{R}) \cong \mathbb{R}^n$. Proceeding under Assumption \ref{ass:ensemble_equivalence} (the asymptotic equivalence of Monic and Kac ensembles), we deduce that the measure of the real image $L$ decays super-exponentially as $\sim e^{-C n^2}$, rendering the asymmetric stratum statistically small for large $n$.

    \item \textbf{Geometric Rarity (General Finite Groups) (Theorem \ref{thm:rarity_real_image_involutions}):} 
    For finite groups acting on vector spaces, we introduce a natural measure on the quotient $Y(\mathbb{R})$ induced by a $G$-invariant Hermitian metric on $X$. Within this framework, we prove that the relative volume of the real image $L$ is exactly $(\#\mathrm{Inv}(G))^{-1}$, identifying the number of involutions as the controlling parameter for the scarcity of the physical domain.

    \item \textbf{Extension to Reductive Groups (Particle Systems) (Theorem \ref{thm:relative_volume_real_image_shape}):} 
    We extend the metric framework to the reductive group action governing particle systems, $G = O(d,\mathbb{C}) \times S_n$. Using a Hermitian metric invariant under the maximal compact subgroup $K = S_n \times O(d, \mathbb{R})$, we prove that the physically realizable domain $L$ constitutes a vanishingly small fraction of the ambient quotient, with relative volume decaying as $\sim 2^{-\min(d,n)} \cdot (\#\mathrm{Inv}(S_n))^{-1}$.
\end{enumerate}

\noindent\textbf{Unified Algebraic Perspective.} 
Beyond these specific estimates, we observe (Remark \ref{rem:galois_cohomology}) that the volume constraints derived above can be unified via Galois Cohomology: the physical image corresponds to the trivial cohomology class, while the ambient space is populated by non-trivial real forms.
\paragraph{II. Empirical Findings.}
Complementing the theoretical analysis, we report the following experimental result:
\begin{itemize}
    \item \textbf{Universality of Energetic Ordering (Regime II):} 
    We document extensive numerical evidence (Appendix \ref{app:boundary_experiments}) showing that the ``Regime II'' phenomenon—where symmetry increases as energy decreases—is not limited to physical potentials. We establish that this monotonic ordering emerges generically in random $S_n$-invariant polynomial landscapes, identifying it as a structural feature of invariant optimization rather than a consequence of specific physical forces.
\end{itemize}

\paragraph{III. Proposed Mechanisms.}
Based on the geometric rarity and the empirical ordering, we propose two driving mechanisms:
\begin{enumerate}
    \item \textbf{Regime I: The ``Empty Interior'' (No Critical Points in $L^\circ$).} 
    For non-convex potentials of moderate complexity (where the number of critical points does not overwhelm the super-exponential volume decay of $\mu(L)$), we posit that critical points are statistically excluded from the smooth locus of $L^\circ$. This forces critical points to reside on the singular locus, resulting in prevalent symmetry.

    \item \textbf{Regime~II (Energetic ordering by symmetry).} We formulate the \emph{Active Constraint hypothesis}: due to the metric rarity of the real image $L$, the physical domain acts as a small, bounded manifold subject to the ambient potential $\tilde{f}$. Consequently, the landscape restricted to $L$ is dominated by a coherent global gradient rather than local fluctuations. This structural configuration localizes the lowest energy states within the high-codimension strata of the boundary $\partial L$. We interpret the ``funnel'' topography of Lennard-Jones clusters simply as the manifestation of such a global gradient oriented toward the boundary strata of high symmetry.
\end{enumerate}

\subsection{Related Work}

The question of the origin of symmetry in physical systems has engaged researchers since the earliest days of mathematical physics. As early as 1894, Pierre Curie proposed the guiding principle that ``the symmetry of the effect is the maximal symmetry compatible with the existence of the phenomenon'' \cite{Curie1894}.
In the context of discrete geometry and optimization, Cohn \cite{Cohn2010} reframes this challenge as the search for the ``genetics of regular figures''---a term coined by Fejes T\'oth to describe the economy principle that generates order out of chaotic sets.
Cohn observes that while the algebraic classification of symmetry groups is well developed, a satisfactory explanation of this generative mechanism remains open.

He further cautions that the intuition that symmetric problems necessarily possess symmetric solutions is often misleading; a prominent example is the Kelvin Conjecture regarding space partition \cite{Kelvin1887}, which stood for over a century before being refuted by the less symmetric Weaire-Phelan structure \cite{Weaire1994}.

This search for a ``genetic'' mechanism is motivated by robust empirical evidence indicating that complex systems exhibit a statistical preference for symmetry. In chemical physics, the ``Epikernel Principle'' formulated by Ceulemans and Vanquickenborne \cite{Ceulemans1989} suggests that when a symmetric system becomes unstable (e.g., via the Jahn-Teller effect), symmetry breaking is not total. Instead, the system prefers to descend into maximal subgroups (epikernels), localizing in high-symmetry strata rather than collapsing to complete asymmetry.

The theoretical justification for the existence and stability of these symmetric critical points was established through two complementary mathematical pillars. First, Palais formulated the Principle of Symmetric Criticality'' (PSC), guaranteeing that critical points found within a symmetric subspace are valid critical points of the global function \cite{Palais1979}. While PSC ensures existence, the analytic basis for \textit{local stability} was laid by Waterhouse \cite{Waterhouse1983}. Responding to historical paradoxes, Waterhouse formalized the intuition known as the Extended Purkiss Principle.'' He proved that when the tangent space is an irreducible representation of the stabilizer subgroup, the Hessian matrix at the critical point is constrained to be a scalar multiple of the identity. This algebraic insight precludes the existence of ``mixed directions'' (some ascending, some descending), thereby preventing symmetric points from becoming saddle points and granting them inherent stability as minima or maxima.

A foundational hypothesis regarding the structure of the ground state in Grand Unified Theories is \textbf{Michel's Conjecture} 
\cite{Michel1980}. It posits that for an irreducible real representation $X(\mathbb{R})$ of a compact group $G$, 
the absolute minima of a any $G$-invariant quartic potential are attained at configurations with \textbf{maximal isotropy subgroups} (stabilizers).
The geometric mechanism underlying this phenomenon was systematically developed by Abud and Sartori \cite{Abud1983}, who analyzed the stratification of the physical domain $L = \pi(X(\mathbb{R}))$. They established the geometric framework of the orbit space to derive rigorous sufficient conditions under which the conjecture holds, demonstrating that within these constraints, the minima are forced to lie on the boundary strata.
Kim \cite{Kim1984} explicitly constructed the orbit spaces for low-dimensional representations. He demonstrated that whenever the potential is monotonic with respect to the orbit parameters, the absolute minimum is attained at the vertices of the orbit space, which correspond to maximal isotropy subgroups.

Similarly, in the context of discrete geometry---primarily concerning point configurations on the sphere---Cohn and Kumar \cite{CohnKumar2007} formulated the rigorous concept of \textbf{universal optimality}. 
Consider a configuration of $N$ points $x = (x_1, \dots, x_N)$ on the unit sphere $S^{d-1}$. The total potential energy is defined as a functional $E_f: (S^{d-1})^N \to \mathbb{R}$ given by the pairwise sum:
\[
E_f(x) = \sum_{1 \le i < j \le N} f(\|x_i - x_j\|^2).
\]
Universal optimality requires a configuration to minimize $E_f$ globally, simultaneously for the entire class of \textbf{completely monotonic} interaction functions.

While explicit counterexamples refute the universality of Michel's conjecture \cite{AbudCounter1984}, and universal optima are guaranteed only in rare cases, the statistical tendency for high-symmetry minima persists. A key distinction in scope exists between these approaches and the present work. While the Michel-Abud-Sartori program and Cohn's framework focus on deterministic guarantees for \textit{maximal} or \textit{universal} symmetry, our framework utilizes the \textit{metric rarity} of the real image to explain the \textit{statistical} preference for \textit{high} (not necessarily maximal) symmetry. This mechanism operates generically across broad families of potentials, outside the strict constraints required for universal optimality.

This preference for order extends to the thermodynamic limit ($N \to \infty$). As reviewed by Blanc and Lewin \cite{Blanc2015}, the ``Crystallization Conjecture'' posits that interacting particles spontaneously arrange in periodic structures at low temperatures. While rigorous proofs---such as Theil's for Lennard-Jones-like potentials \cite{Theil2006}---are rare and analytically demanding, the empirical prevalence of crystallization underscores the universality of symmetry selection.

A general statistical explanation for these observations was proposed by Wales, who extended the analysis to include ``near-symmetry'' \cite{wales1997global}. In an extensive survey of atomic clusters and molecular models, Wales identified a strong correlation between high (exact or approximate) symmetry and energetic extremality. He argued that such structures are composed of fewer distinct energy contributions, effectively reducing the number of independent random variables summing to the total energy. This reduction leads to a larger variance in the energy distribution, thereby populating the distribution tails. Consequently, symmetric and near-symmetric configurations appear with higher probability at the energetic extremes (both minima and maxima).
We explicitly discuss the validity of this hypothesis and its integration into our geometric framework in Appendix \ref{app:mechanism_disentanglement}.

Despite these significant theoretical and empirical advances, a comprehensive answer to the fundamental question of symmetry selection remains elusive. While the existing frameworks successfully address specific aspects---ensuring the existence of symmetric critical points, verifying their local stability, or rationalizing their statistical presence in the tails of distributions---none offers a complete, unified explanation for the prevalence of symmetric critical points in global optimization. The mechanism by which a complex system selects a symmetric global minimum from a configuration space in which symmetric points have measure zero remains not fully understood. In this work, we aim to bridge this gap by constructing a rigorous deterministic mechanism that serves as the geometric substrate for this statistical tendency.

\subsection{Paper Organization}
The remainder of this paper is organized as follows.

In \textbf{Section \ref{sec:algebraic_framework}}, we establish the algebraic foundation, defining the general $G$-variety setup, the quotient space construction, and the Principle of Symmetric Criticality.
\textbf{Section \ref{sec:geometry_quotient}} provides the geometric dictionary connecting algebra to topology. 
We prove that the differential of the quotient map is surjective if and only if the stabilizer is 
trivial (Theorem \ref{thm:surjectivity_diff}) and identify configurations with even order stabilizer with the boundary of the real image (Theorem \ref{thm:boundary_parity}).
In \textbf{Section \ref{sec:metric_rarity}}, we quantify the "metric rarity" of the real image. We begin by analyzing the quotient of $S_n$ acting on $\mathbb{R}^n$, utilizing probabilistic results regarding polynomials with real roots. We then define a natural measure for general finite group actions on vector spaces to calculate the rarity of the real image, and discuss how this rarity applies recursively across fixed-point strata.
\textbf{Section \ref{sec:LJ_setup}} extends this framework to physical particle systems. We formalize the Shape Space and the Total Quotient, proving that the volume of physically realizable configurations decays exponentially within the ambient quotient (Theorem \ref{thm:relative_volume_real_image_shape}).
In \textbf{Section \ref{sec:empirical_implications}}, we apply our theoretical framework to interpret empirical observations. We map the experimental setups from the works of Arjevani et al. and Schneider to our algebraic model, demonstrating how the metric rarity of the real image explains the prevalence of symmetry (``Regime I'') in these landscapes.
Finally, \textbf{Section \ref{sec:conclusion}} summarizes the findings and discusses the implications for Energy Landscape Theory and future research directions.

Appendices are dedicated to supporting analysis: \textbf{Appendix \ref{app:capacity}} provides a heuristic counting baseline deriving from the Principle of Symmetric Criticality, and \textbf{Appendix \ref{app:boundary_experiments}} presents extensive additional numerical experiments confirming the robustness of the Regime II phenomenon. Finally, \textbf{Appendix \ref{app:mechanism_disentanglement}} critically analyzes the drivers of the "Active Constraint," disentangling intrinsic geometric effects from the variance-based mechanism proposed by Wales.

\FloatBarrier
\section{The Algebraic Framework}
\label{sec:algebraic_framework}
To formalize the geometric stratification described in the introduction, we define the algebraic structure of the configuration space and its quotient.

\begin{definition}[The $G$-Variety Setup]
\label{def:variety_setup}
Let $X$ be an irreducible complex affine algebraic variety defined over $\mathbb{R}$, equipped with a faithful action of a finite group $G$ by real automorphisms. We denote by $X(\mathbb{R})$ the set of real points of $X$.
\end{definition}

\begin{definition}[Stabilizer Subgroup]
\label{def:stabilizer}
For any point $x \in X$, the \textbf{stabilizer subgroup} is defined as:
\[
G_x = \{g \in G \mid g \cdot x = x\}.
\]
\end{definition}

% --- Definition 3: Fixed Point Subspace ---
\begin{definition}[Fixed Point Subspace]
\label{def:fixed_point_space}
For any subgroup $H \subseteq G$, the \textbf{fixed point subspace} (or symmetric stratum) is the closed subvariety defined by:
\[
X^H = \{x \in X \mid h \cdot x = x, \quad \forall h \in H\}.
\]
We denote the real locus of this subspace by $X(\mathbb{R})^H = X^H \cap X(\mathbb{R})$.
\end{definition}

% --- Definition 4: Invariant Functions ---
\begin{definition}[Invariant Functions]
\label{def:invariant_function}
Let $f: X \to \mathbb{C}$ be a regular function defined over $\mathbb{R}$. We say $f$ is \textbf{$G$-invariant} if $f(g \cdot x) = f(x)$ for all $g \in G, x \in X$.
We denote the ring of such functions by $\mathbb{R}[X]^G$.
\end{definition}

\begin{definition}[The Quotient Space]
\label{def:quotient_space}
The \textbf{quotient space} $Y = X \sslash G$ is constructed as the affine variety corresponding to the ring of invariants, $Y := \Spec(\mathbb{R}[X]^G)$. The inclusion $\mathbb{R}[X]^G \hookrightarrow \mathbb{R}[X]$ induces the categorical quotient map $\pi: X \to Y$.
\end{definition}

Finally, the following principle ensures that searching for symmetric critical points is equivalent to optimizing directly on the stratum $X(\mathbb{R})^H$.

% --- Theorem: PSC ---
\begin{theorem}[Principle of Symmetric Criticality \cite{Palais1979}]
\label{thm:PSC}
Let $f: X(\mathbb{R}) \to \mathbb{R}$ be a smooth $G$-invariant function and let $H \subseteq G$ be a subgroup.
If a point $x \in X(\mathbb{R})^H$ is a critical point of the restriction $f|_{X(\mathbb{R})^H}$, then $x$ is a critical point of $f$ on the entire manifold $X(\mathbb{R})$.
\end{theorem}

Theorem \ref{thm:PSC} guarantees the existence of symmetric critical points but implies nothing about their abundance.

\begin{remark}[Heuristic Counting Baseline]
\label{rem:counting_baseline}
While Theorem \ref{thm:PSC} guarantees the existence of symmetric critical points, it also provides a mechanism for estimating their statistical abundance. By combining this principle with standard algebraic root-counting heuristics (such as the "square-root law"), one can derive a baseline prediction for the ratio of symmetric to asymmetric critical points in random landscapes. We provide a detailed derivation of this counting argument and its parameter thresholds in Appendix \ref{app:capacity}.
\end{remark}

\section{The Geometry of the Quotient Map: Singularities and Boundaries}
\label{sec:geometry_quotient}
The statistical preference of invariant potentials for symmetric configurations can be traced to the geometry of the real image $L = \pi(X(\mathbb{R}))$. In this section, we establish a correspondence between the algebraic properties of a point (its stabilizer) and its topological placement within the quotient (interior vs.\ boundary).

We begin by characterizing the "smooth" behavior of the quotient map. The following theorem establishes that away from symmetric configurations, the map behaves like a local diffeomorphism and conversely, singularities in the map correspond precisely to points with non-trivial stabilizers.

\begin{proposition}[Surjectivity of the Quotient Differential; see e.g., \cite{brion_luna}]
\label{thm:surjectivity_diff}
Let $X$ be an irreducible complex affine algebraic variety defined over $\mathbb{R}$, equipped with a faithful action of a finite group $G$. Let $x \in X(\mathbb{R})$ be a smooth point and let $y = \pi(x)$ be its image under the quotient map $\pi: X \to Y := X/G$.
The differential $d\pi_x: T_xX(\mathbb{R}) \to T_yY(\mathbb{R})$ is surjective if and only if the stabilizer $G_x$ is trivial.
\end{proposition}
\begin{proof}
    \textbf{Step 1: Reduction to the Tangent Space (Luna's Slice Theorem)}
    
    Since $G$ is finite and $x \in X(\mathbb{R})$ is a smooth point, we apply Luna's Slice Theorem~\cite[Theorem 5.3]{brion_luna}. Let $V = T_x X$ be the tangent space at $x$. 
    Let $V(\mathbb{R})$ denote the set of real points of this vector space.
    
    According to the Slice Theorem, there exists a $G_x$-invariant locally closed subvariety $S \subset X$ containing $x$, such that the natural $G$-morphism $\psi: G \times^{G_x} S \to X$ is \textbf{excellent} (see \cite[Definition 5.7]{brion_luna}). Since $x$ is a smooth point, we can identify $S$ with a $G_x$-invariant neighborhood $U$ of $0$ in $V$. 
    
    This yields the following commutative diagram of real algebraic varieties and morphisms, where the horizontal maps are étale:
    \[
    \begin{tikzcd}[row sep=large, column sep=huge]
        G \times^{G_x} U \arrow[r, "\psi"] \arrow[d, "q"'] 
        & X \arrow[d, "\pi"] \\
        U /\!/ G_x \arrow[r, "\overline{\psi}"] 
        & X /\!/ G
    \end{tikzcd}
    \]
    Here, $q$ is induced by the projection to the quotient of the linear action.
    Since $\overline{\psi}$ is an étale morphism of real varieties (a property implied by $\psi$ being excellent~\cite[Definition 5.7(i)]{brion_luna}), it induces a local homeomorphism on the set of real points.
    
    The horizontal maps $\psi$ and $\overline{\psi}$ are étale \cite[Theorem 5.3]{brion_luna}. Étale morphisms induce isomorphisms on the tangent spaces. Therefore, the surjectivity of $d\pi_x$ is equivalent to the surjectivity of the differential of the quotient map $dq_0: T_0V \to T_{q(0)}(V/\!/G_x)$ at the origin.

    \medskip
    \textbf{Step 2: Analysis of the Linear Quotient.}
    The problem reduces to a statement about linear representations: Let $V = T_xX$ and let $q: V \to V/\!/G_x$ be the quotient by the finite linear group $G_x$. We must show that $dq_0$ is surjective if and only if $G_x$ is trivial.

    Let $\mathbb{C}[V]$ be the algebra of polynomial functions on $V$, and $\mathbb{C}[V]^{G_x}$ be the algebra of invariants defining $V/\!/G_x$. Let $\mathfrak{m}_0 \subset \mathbb{C}[V]$ and $\mathfrak{n}_0 \subset \mathbb{C}[V]^{G_x}$ be the maximal ideals corresponding to the origin in $V$ and its image in the quotient, respectively.
    The Zariski tangent spaces are dual to the cotangent spaces: $T_0V \cong (\mathfrak{m}_0/\mathfrak{m}_0^2)^*$ and $T_{q(0)}(V/\!/G_x) \cong (\mathfrak{n}_0/\mathfrak{n}_0^2)^*$.
    
    The dual of the differential, $(dq_0)^*: \mathfrak{n}_0/\mathfrak{n}_0^2 \to \mathfrak{m}_0/\mathfrak{m}_0^2$, maps the class of an invariant function $f$ to its linear term at the origin $df_0$. Consequently, the image of $(dq_0)^*$ corresponds precisely to the subspace of $G_x$-invariant linear forms on $V$:
    \[
    \operatorname{Im}((dq_0)^*) \cong (V^*)^{G_x}.
    \]
    
    \medskip
    \textbf{Step 3: Conclusion.}
Note that the rank of the linear map $dq_0$ is equal to the dimension of the image of its dual. Thus:
$$\operatorname{rank}(dq_0) = \dim ((V^*)^{G_x}) = \dim (V^{G_x}).$$
This implies that the image of the differential corresponds exactly to the subspace of fixed points $V^{G_x}$.
On the other hand, for $dq_0$ to be surjective onto the tangent space $T_{q(0)}(V/\!/G_x)$, its rank must be at least the dimension of the quotient variety. 
Since $G$ is finite, $\dim(V/\!/G_x) = \dim V$. Thus, a necessary condition for surjectivity is:
$$\dim (V^{G_x}) \geq \dim V.$$This equality holds if and only if $V^{G_x} = V$, which means the group $G_x$ acts trivially on the tangent space $V$. 
Finally, since $x$ is a smooth point and the action of $G$ on $X$ is faithful, 
the linearized action of the stabilizer $G_x$ on the tangent space $T_xX$ must also be faithful. 
Therefore, the action is trivial if and only if the group itself is trivial:$$G_x = \{1\}.$$
\end{proof}

Theorem \ref{thm:surjectivity_diff} implies that points with trivial stabilizers map to the interior. We now determine which symmetric configurations constitute the topological boundary of the real image. The following result establishes that this distinction is governed solely by the parity of the stabilizer.

\begin{theorem}[Boundary Stratification; adapted from \cite{procesi1985}]
\label{thm:boundary_parity}
    Let $X$ be a complex affine irreducible algebraic variety defined over $\mathbb{R}$, and let $G$ be a finite group acting faithfully on $X$ by real algebraic automorphisms.
    Let $Y = X//G$ be the algebraic quotient with quotient map $\pi:X\to Y$, and let $X(\mathbb{R})$ and $Y(\mathbb{R})$ denote the real points.
    Let $L = \pi(X(\mathbb{R}))\subset Y(\mathbb{R})$ be the real image.
    If $x \in X(\mathbb{R})$ is a smooth point, then $\pi(x) \in \partial L \subset Y(\mathbb{R})$ if and only if the order of the stabilizer $|G_x|$ is even.
\end{theorem}

\begin{proof}
    We adapt the argument of \cite[Corollary~3.9]{procesi1985} to our setting.\\
    \textbf{Step 1: Reduction to the Tangent Space}

    We proceed analogously to the proof of Theorem \ref{thm:surjectivity_diff}. Let $V = T_x X$ denote the tangent space at $x$. 
    By Luna's Slice Theorem, the quotient map $\pi: X \to X \sslash G$ is locally modeled (in the étale topology) by the linear quotient map $r: V \to V \sslash G_x$.

    Since the relating morphisms are étale, they induce local homeomorphisms on the set of real points. Therefore, the topological property of lying on the boundary of the image is preserved:
    \[
    \pi(x) \in \partial \pi(X(\mathbb{R})) \iff 0 \in \partial r(V(\mathbb{R})).
    \]
    Thus, it suffices to prove the statement for the linear representation of the stabilizer $G_x$ on $V$.
    \vspace{1em}
    \textbf{Step 2: The Even Case ($|G_x|$ is even)}
    
    Assume $|G_x|$ is even. By Cauchy's Theorem, there exists an element $\sigma \in G_x$ of order $2$. Since $G$ acts faithfully on irreducible $X$, the action of $\sigma$ on $V(\mathbb{R})$ is non-trivial. Since $\sigma^2 = \text{id}$, we have the eigenspace decomposition of the real vector space $V(\mathbb{R}) = V_+ \oplus V_-$. Since $\sigma \neq \text{id}$, $V_-$ is non-trivial. Fix a non-zero vector $v \in V_-$.
    
    Consider the curve of purely imaginary vectors $v_\epsilon = i \epsilon v$ for $\epsilon \in \mathbb{R}$. We observe:
    \[
        \overline{v_\epsilon} = -i \epsilon v = i \epsilon (-v) = i \epsilon \sigma(v) = \sigma(i \epsilon v) = \sigma(v_\epsilon).
    \]
    Applying the quotient map $r$, which is real and $G$-invariant:
    \[
        \overline{r(v_\epsilon)} = r(\overline{v_\epsilon}) = r(\sigma(v_\epsilon)) = r(v_\epsilon).
    \]
    Thus, $r(v_\epsilon) \in (T_xX // G_x)(\mathbb{R})$.
    
    We claim that for $\epsilon \neq 0$, $r(v_\epsilon) \notin r(V(\mathbb{R}))$. Suppose for contradiction that $r(v_\epsilon) = r(u)$ for some $u \in V(\mathbb{R})$. Since $G$ is finite, the orbits are closed, so this implies $u = g \cdot v_\epsilon$ for some $g \in G_x$. 
    However, $u$ is real and $v_\epsilon$ is purely imaginary. The only vector that is both real and purely imaginary is $0$. Since $v \neq 0$ and $\epsilon \neq 0$, this is a contradiction.
    
    Thus, $r(v_\epsilon)$ is a sequence of real points in the quotient that are not images of real points in $X$, converging to $r(0)$. Therefore, $0 \in \partial r(V(\mathbb{R}))$.
    
    \vspace{1em}
    \textbf{Step 3: The Odd Case ($|G_x|$ is odd)}
    
    Assume $|G_x|$ is odd. Let $y \in (T_xX /\!/ G_x)(\mathbb{R})$. Since $r$ is surjective over $\mathbb{C}$, let $u \in T_xX$ be a preimage of $y$.
    
    Since $y$ is real, $r(\overline{u}) = \overline{r(u)} = \overline{y} = y = r(u)$. This implies that $\overline{u}$ and $u$ lie in the same $G_x$-orbit. Thus, $\overline{u} = gu$ for some $g \in G_x$.
    
    Applying complex conjugation (and noting $\overline{g}=g$ as $G$ acts by real automorphisms):
    \[ u = \overline{\overline{u}} = \overline{gu} = g \overline{u} = g(gu) = g^2 u. \]
    Thus, $u$ is fixed by $g^2$. Since $|G_x|$ is odd, the order of $g$, say $k$, is odd. Therefore, $g^k = \text{id}$. Therefore:
    \[ gu = g^k u = \text{id} \cdot u = u. \]
    
    Therefore, $\overline{u} = gu = u$, so $u \in V(\mathbb{R})$.
    
    This proves that the map $r|_{V(\mathbb{R})}: V(\mathbb{R}) \to (T_xX // G_x)(\mathbb{R})$ is surjective. Hence, $0$ is an interior point of $r(V(\mathbb{R}))$.
\end{proof}

\begin{remark}[Identification of Symmetry and Boundary]
\label{rem:symmetry_boundary_identification}
In the context of $S_n$ invariant polynomials and the $LJ_{13}$ potential the parity condition in Theorem \ref{thm:boundary_parity} allows us to effectively identify symmetric configurations with the boundary of the real image $\partial L$ in our primary settings of interest.

First, consider the action of the permutations group $S_n$ on $\mathbb{C}^n$. Since $S_n$ is a reflection group, any non-trivial stabilizer $G_x$ contains a transposition (an element of order 2). Consequently, $|G_x|$ is always even, and thus every symmetric configuration maps to the boundary $\partial L$.

Second, in broader physical contexts where odd-order stabilizers are theoretically possible (e.g., cyclic groups $C_3, C_5$), empirical evidence suggests they are statistically negligible. For instance, in the $LJ_{13}$ database, out of approximately 30,517 recorded critical points, only 3 possess a stabilizer of odd order (specifically order 3) distinct from the identity.

Therefore, throughout this paper, we will adopt the convention of using the terms \textbf{symmetric configuration} and \textbf{boundary point} interchangeably.
\end{remark}

\section{Metric Rarity in Vector Space Quotients}
\label{sec:metric_rarity}

\subsection{Metric Rarity in the Symmetric Group Quotient}

We begin by quantifying the measure of the real image in the canonical case of the symmetric group $S_n$ acting on the configuration space $\mathbb{R}^n$.
In this setting, the quotient map $\pi: \mathbb{C}^n \to \mathbb{C}^n/S_n \cong \mathbb{C}^n$ is realized by the elementary symmetric polynomials (Vieta's map):
\[
\pi(x_1, \dots, x_n) = (e_1(x), \dots, e_n(x)).
\]
We identify the target space $Y(\mathbb{R})$ with the space of monic real polynomials $P(t) = t^n - e_1 t^{n-1} + \dots + (-1)^n e_n$.
Consequently, the real image $L = \pi(\mathbb{R}^n)$ corresponds precisely to the set of coefficients for which \textbf{all roots are real}.

We now demonstrate that for large $n$, the subset $L$ of real-rooted polynomials has vanishing measure relative to the full coefficient space $Y(\mathbb{R})$.

To formalize this, we rely on the standard Gaussian model for random polynomials.
Our quotient space $Y(\mathbb{R})$ corresponds to the ensemble of \emph{monic} polynomials $P(t) = t^n + \sum_{k=1}^n \xi_k t^{n-k}$ with Gaussian coefficients.
The classical theory typically considers the \emph{Kac ensemble}, where the leading coefficient $\xi_0$ is also random.
Here, we proceed under the assumption that fixing a single coefficient (the leading term) does not alter the asymptotic decay rate of the probability of having all real roots.

\begin{assumption}[Ensemble Equivalence]
\label{ass:ensemble_equivalence}
We assume that the probability that a monic Gaussian polynomial has all real roots decays at the same super-exponential rate as in the Kac ensemble.
\end{assumption}

We now prove the main bound for the standard case.

\begin{theorem}[Super-Exponential Decay for Kac Polynomials]
\label{thm:rarity_kac}
Let $Q_n(t) = \sum_{k=0}^n \xi_k t^k$ be a random Kac polynomial with i.i.d.\ standard normal coefficients $\xi_k \sim \mathcal{N}(0,1)$.
The probability that all roots of $Q_n$ are real decays super-exponentially:
\begin{equation}
\limsup_{n \to \infty} \frac{1}{n^2} \log \mathbb{P}(\text{all roots of } Q_n \text{ are real}) < 0.
\end{equation}
\end{theorem}

\begin{proof}
Let $\mu_n = \frac{1}{n} \sum_{z : Q_n(z)=0} \delta_{z}$ denote the empirical measure of the roots of $Q_n$.
We invoke the Large Deviation Principle (LDP) established by Butez \cite[Theorem 1.5]{Butez2016} (extending Zeitouni--Zelditch \cite{ZZ2010}). The sequence of random measures $(\mu_n)$ satisfies an LDP in the space of probability measures $\mathcal{M}_1(\mathbb{C})$ with speed $\beta_n = n^2$ and a good rate function $\tilde{I}$.

The rate function $\tilde{I}(\mu)$ possesses a unique minimizer: the uniform measure on the unit circle $\nu_{S^1}$ (see \cite{Butez2016}, Section 1).
The event that all roots are real corresponds to the set of measures supported on the real line:
\[
A_{\mathbb{R}} = \{ \mu \in \mathcal{M}_1(\mathbb{C}) : \text{supp}(\mu) \subseteq \mathbb{R} \}.
\]
Since the equilibrium measure $\nu_{S^1}$ is supported on the circle and not on the line, $\nu_{S^1} \notin A_{\mathbb{R}}$. Since $A_{\mathbb{R}}$ is a closed set in the weak topology and does not contain the unique minimizer of the good rate function $\tilde{I}$, the infimum of the rate function over this set is strictly positive:
\[
\inf_{\mu \in A_{\mathbb{R}}} \tilde{I}(\mu) =: C > 0.
\]
The LDP upper bound then implies:
\[
\limsup_{n \to \infty} \frac{1}{n^2} \log \mathbb{P}(\mu_n \in A_{\mathbb{R}}) \le -\inf_{\mu \in A_{\mathbb{R}}} \tilde{I}(\mu) = -C.
\]
\end{proof}

We apply this geometric decay to the statistical analysis of critical points.
Let $f$ be an $S_n$-invariant polynomial that descends to a polynomial $\tilde{f}: Y(\mathbb{R}) \to \mathbb{R}$ via the relation $f = \tilde{f} \circ \pi$.
If $\tilde{f}$ is a polynomial of degree $n$, then by Bézout's Theorem, its gradient system $\nabla \tilde{f} = 0$ possesses at most $n^n$ complex critical points.
Recall from Theorem \ref{thm:surjectivity_diff} that an asymmetric critical point of $f$ corresponds uniquely to a critical point of $\tilde{f}$ lying in the smooth locus of the real image $L^\circ$.
The problem thus reduces to estimating the probability that any of these $n^n$ candidates lies in the domain $L$.

\begin{corollary}[Vanishing of Asymmetric Critical Points]
\label{cor:rarity_monic}
    Under Assumption \ref{ass:ensemble_equivalence}, the probability that a generic $S_n$-invariant function $f = \tilde{f} \circ \pi$, where $\tilde{f}$ is of degree $n$, possesses any asymmetric critical point tends to zero as $n \to \infty$.
\end{corollary}

\begin{proof}
    Let $N_{\text{crit}}$ denote the number of asymmetric critical points of $f$. This is equal to the number of critical points of $\tilde{f}$ that fall into the real image $L$.
    Combining the bound on the number of candidates ($n^n$) with the volume decay from Theorem \ref{thm:rarity_kac} ($\sim e^{-C n^2}$), the expected number of such points satisfies:
    \[
    \mathbb{E}[N_{\text{crit}}] \le n^n \cdot \mathbb{P}(y \in L) \approx e^{n \log n} \cdot e^{-C n^2} = e^{n \log n - C n^2}.
    \]
    Since the quadratic decay $n^2$ dominates the growth $n \log n$, the expected number tends to zero.
    By Markov's inequality, $\mathbb{P}(N_{\text{crit}} \ge 1) \le \mathbb{E}[N_{\text{crit}}]$, and thus the probability of observing even a single asymmetric critical point vanishes.
\end{proof}

\subsection{Generalization to Finite Group Actions}

The Gaussian analysis above relies on the specific identification of the quotient space with polynomial coefficients. 
To extend this geometric intuition to general linear representations—and subsequently to the physical models discussed in the introduction—we require a metric framework intrinsic to the variety $X$.
In this section, we define a natural measure on the quotient space induced by the $G$-invariant metric on $X$, allowing us to quantify the rarity of the real image for any finite group.

% --- Case 2: Geometric Measure ---
\begin{definition}[Induced Geometric Measure on $Y(\mathbb{R})$]
\label{def:algebraic_involutions_measure}
    Let $X$ be $n$ dimensional faithful complex $G$-representation defined over $\mathbb{R}$. Let $h$ be a $G$-invariant Hermitian metric on $X$ with an underlying real metric $g_h$. We define the measure $\operatorname{vol}_{Y(\R)}$ on $Y$ via the pullback to the covering space:
    \[
    \operatorname{vol}_{Y(\R)}(U) := \frac{1}{|G|} \mathcal{H}^n_{(X, g_h)}\left(\pi^{-1}(U)\right) \quad \text{for any measurable } U \subset Y(\R),
    \]
    where $\mathcal{H}^n$ is the $n$-dimensional Hausdorff measure.
\end{definition}

\begin{theorem}[Real Image Volume Ratio]
    \label{thm:rarity_real_image_involutions}
    Let $X$ be a faithful complex $G$-representation defined over $\mathbb{R}$. Let $h$ be a $G$-invariant Hermitian metric on $X$ with an underlying real metric $g_h$. 
    Let $B_{Y(\R)}(r) \subset Y(\mathbb{R})$ be a metric ball of radius $r > 0$ centered at the origin $\pi(0)$. The relative volume of the real image is constant and given by:
    \[
    \frac{\operatorname{vol}_{Y(\R)}\left(L \cap B_{Y(\R)}(r)\right)}{\operatorname{vol}_{Y(\R)}\left(B_{Y(\R)}(r)\right)} = \frac{1}{\#\mathrm{Inv}(G)},
    \]
    where $\#\mathrm{Inv}(G) = \#\{\sigma \in G \mid \sigma^2 = \mathrm{id}\}$.
\end{theorem}

\begin{proof}
The proof follows a geometric argument communicated by W.~Sawin \cite{sawin_private}.
  A point $y\in Y(\R)$ is real if and only if $y=\overline y$. Lifting this to $X$, if $\pi(x)=y$ then $\overline x=\sigma x$ for some $\sigma\in G$ satisfying $\sigma^2=\mathrm{id}$. Define the twisted real subspaces:
\[
  U_\sigma:=\{x\in X:\ \overline x=\sigma x\}=\mathrm{Fix}(\sigma\circ\mathrm{conj}).
\]
Recall that $U_{\mathrm{id}} = X(\R)$.
Since the union is finite, up to sets of $\mathcal{H}^n$-measure zero (intersections of lower dimension), we have:
\[
  \pi^{-1}(Y(\R)) \approx \bigsqcup_{\sigma\in\mathrm{Inv}(G)} U_\sigma.
\]
Since $h$ is $G$-invariant, $U_\sigma$ is a real vector space of dimension $n$, and $h$ restricts to an inner product on it. Therefore, there exists a unitary transformation $T_\sigma \in U(X)$ such that $T_\sigma(U_{\mathrm{id}})=U_\sigma$. 
Consequently, $T_\sigma$ is a global isometry of the metric space defined by $h$. Thus, it preserves the Hausdorff measure $\mathcal{H}^n$, and therefore:
\[
  \mathcal{H}^n(U_\sigma\cap B_{X}(r)) = \mathcal{H}^n(T_\sigma(U_{\mathrm{id}}\cap B_{X}(r))) = \mathcal{H}^n(U_{\mathrm{id}}\cap B_{X}(r)) =: C_r.
\]
The rest of the proof follows by summation:
\[
  \vol_{Y(\R)}(B_{Y(\R)}(r))=\frac{1}{|G|}\sum_{\sigma\in\mathrm{Inv}(G)}C_r=\frac{\#\mathrm{Inv}(G)}{|G|}\,C_r.
\]
For the numerator, on the principal stratum the conjugation argument shows $U_\sigma\cap \pi^{-1}(\pi(X(\R)))=\varnothing$ when $\sigma\neq\mathrm{id}$, leaving only $U_{\mathrm{id}}$. Thus
\[
  \vol_{Y(\R)}\big(\pi(X(\mathbb{R}))\cap B_{Y(\R)}(r)\big)=\frac{1}{|G|}\,C_r.
\]
Taking the quotient yields $1/\#\mathrm{Inv}(G)$, uniformly in $r>0$.
\end{proof}

Applying this general geometric formula back to the case of the symmetric group allows us to compare the intrinsic geometric rarity against the probabilistic Gaussian rarity derived in Theorem \ref{thm:rarity_kac}.

\begin{corollary}[Specialization to $S_n$]
  \label{cor:Inv_Sn_decay}
For $G=S_n$ acting on $X(\mathbb{R})=\mathbb{R}^n$,
\[
\frac{\vol_{Y(\R)}\big(\pi(\mathbb{R}^n)\cap B_{Y(\R)}(r)\big)}{\vol_{Y(\R)}\big(B_{Y(\R)}(r)\big)}
=\frac{1}{\,\#\mathrm{Inv}(S_n)\,},\qquad
\#\mathrm{Inv}(S_n)=\sum_{k=0}^{\lfloor n/2\rfloor}\frac{n!}{2^k\,k!\,(n-2k)!}.
\]
Asymptotically, we have $\#\mathrm{Inv}(S_n) = O(n^{n/2})$.
Consequently, there exists a constant $C$ such that 
\[
\frac{1}{\,\#\mathrm{Inv}(S_n)\,} = O(e^{-Cn\log n}).
\]
\end{corollary}

\begin{remark}
\label{rem:compare_measures}
Comparison with the Gaussian case validates the use of our induced geometric measure.
The Gaussian measure, which represents a standard "unbiased" distribution on the coefficient space, under Assumption \ref{ass:ensemble_equivalence} dictates an extremely rapid super-exponential decay for the volume of the real image ($\sim e^{-Cn^2}$).
In contrast, our geometric measure yields a slower decay ($\sim e^{-Cn \log n}$).
This suggests that the geometric measure does not artificially suppress the size of the real image; in fact, it attributes to it a relative volume strictly \textbf{larger} than the standard probabilistic baseline.
\end{remark}

\subsection{Iteration Along Fixed-Point Strata}
\label{subsec:iteration_strata}

The geometric rarity established in Theorem \ref{thm:rarity_real_image_involutions} is not limited to the global quotient. The logic extends recursively to the lattice of fixed-point subspaces. To quantify the rarity of symmetric critical points, we must treat each stratum $X^H$ as a quotient space in its own right, governed by the effective action of the normalizer.

\begin{definition}[Effective Action and Stratum Quotient]
\label{def:normalizer_effective}
For a subgroup $H \subseteq G$, the normalizer $N_G(H)$ acts on the fixed-point subspace $X^H$. Since $H$ acts trivially on this subspace, the action factors through the \textbf{effective group}:
\[
\overline{G} := N_G(H)/H.
\]
Consequently, the geometry of the stratum is captured by the local quotient map $\pi\mid_{X^H}$ onto the \textbf{stratum quotient} $Y^{(H)}$:
\[
\pi: X(\mathbb{R})^H \longrightarrow Y^{(H)}(\mathbb{R}) := X(\mathbb{R})^H \sslash \overline{G}.
\]
\end{definition}

We generalize the metric setup by restricting the $G$-invariant Hermitian metric $h$ to the subspace $X^H$ and defining the induced measure $\operatorname{vol}_{Y^{(H)}(\R)}$ on the stratum quotient analogously to Definition \ref{def:algebraic_involutions_measure}, replacing $G$ with $\overline{G}$ and $X$ with $X^H$.

\begin{lemma}[Relative Volume of Real Image in Strata]
\label{lem:stratum_volume}
Let $L_H = \pi(X(\mathbb{R})^H)$ be the real image within the quotient of the $H$-stratum. For any ball $B_{Y^{(H)}(\R)}(r)$ of radius $r > 0$ in the stratum quotient, the relative volume is given by:
\[
\frac{\operatorname{vol}_{Y^{(H)}(\R)}\left(L_H \cap B_{Y^{(H)}(\R)}(r)\right)}{\operatorname{vol}_{Y^{(H)}(\R)}\left(B_{Y^{(H)}(\R)}(r)\right)} = \frac{1}{\#\mathrm{Inv}(\overline{G})},
\]
where $\#\mathrm{Inv}(\overline{G})$ is the number of involutions in the effective group $\overline{G}$.
\end{lemma}

\begin{proof}
The proof is identical to that of Theorem \ref{thm:rarity_real_image_involutions}, applied to the faithful action of the finite group $\overline{G}$ on the variety $X^H$. The measure $\operatorname{vol}_{Y^{(H)}(\R)}$ is defined via the pullback to $X^H$, and the decomposition into twisted real subspaces follows the same algebraic structure with respect to the group $\overline{G}$.
\end{proof}

\begin{remark}[Monotonicity and the Emergence of High Symmetry]
\label{rem:monotonicity_strata}
This result establishes a monotonicity of "probability" across strata. 
Typically, as the stabilizer $H$ increases in size, the effective group $\overline{G} = N_G(H)/H$ becomes smaller. Consequently, the number of involutions $\#\mathrm{Inv}(\overline{G})$ decreases.
Since the relative volume of the real image is proportional to $1/\#\mathrm{Inv}(\overline{G})$, the real image becomes less rare in strata associated with higher symmetries.
This provides a geometric justification for the recursive emergence of high symmetry: a generic critical point is statistically unlikely to reside in strata where $\overline{G}$ is large (due to the rarity of the real image), and is thus constrained to strata where $\overline{G}$ is small.
\end{remark}

\section{The Physical Setup: Particle Systems and Shape Spaces}
\label{sec:LJ_setup}

In this section, we formalize the algebraic framework for atomic clusters in dimension $d$. Our primary object of study is a system of $n$ particles governed by a pairwise interaction potential—an energy function depending solely on inter-particle distances. Such potentials are inherently invariant under the full Euclidean group (translations and rotations) as well as particle relabeling.

To eliminate the trivial translational degrees of freedom, we fix the center of mass at the origin. This restricts the system to the following configuration space:

\begin{definition}[Configuration Space]
\label{def:V_configurations}
Let $\mathcal{V} \cong \mathbb{C}^{d(n-1)}$ denote the affine space of centered particle configurations:
\[
    \mathcal{V} := \left\{ x = (x_1,\dots,x_n) \in (\mathbb{C}^d)^n \;\middle|\; \sum_{i=1}^n x_i = 0 \right\}.
\]
\end{definition}

\begin{definition}[Pairwise Interaction Potentials]
\label{def:general_potential}
We consider rational potentials $V: \mathcal{V} \dashrightarrow \mathbb{C}$ of the form
\[
  V(x) := \sum_{1 \le i < j \le n} \phi\left( r_{ij}^2 \right),
\]
where $\phi$ is a rational function and $r_{ij}^2 = \langle x_i - x_j, x_i - x_j \rangle$ is defined via the standard \textbf{symmetric bilinear form} on $\mathbb{C}^d$ (the algebraic extension of the real Euclidean metric).

In particular, the \textbf{Lennard-Jones potential} $V_{LJ}$ corresponds to the specific choice of the rational function:
\[
\phi_{LJ}(z) = \frac{1}{z^6} - \frac{1}{z^3} 
\]
acting on the squared distance $z = r_{ij}^2$.
\end{definition}

\begin{definition}[Group Actions]
\label{def:actions}
With translational symmetry removed, the remaining invariances induce two commuting group actions on $\mathcal{V}$:
\begin{enumerate}
    \item The complex orthogonal group $O(d,\mathbb{C})$, acting by left multiplication (rotations).
    \item The symmetric group $S_n$, acting by right multiplication (relabeling).
\end{enumerate}
\end{definition}

To isolate the intrinsic geometry of the cluster, we first quotient out the continuous rotational symmetries.

\begin{definition}[Shape Space and Physical Cone]
\label{def:shape_space}
The \textbf{Shape Space} is the categorical quotient of the configuration space by the orthogonal group:
\[
    X_{\shape} := \mathcal{V} \sslash O(d,\mathbb{C}).
\]
Let $q : \mathcal{V} \to X_{\shape}$ denote the quotient map.
Strictly real configurations map to a subset of the real locus of the quotient. We define the \textbf{Physical Shape Space} as this image:
\[
    X_{\shape}^+ := q(\mathcal{V}(\mathbb{R})) \subset X_{\shape}(\mathbb{R}).
\]
\end{definition}

Since rotations and permutations commute, the action of $S_n$ descends uniquely to a faithful algebraic action on $X_{\shape}$, given by $q(x) \cdot \sigma := q(x \cdot \sigma)$. We obtain the total quotient space by removing this remaining symmetry.

\begin{definition}[Total Quotient and Descended Potential]
\label{def:Y_shape}
We define the total quotient space by:
\[
    Y_{\shape} := X_{\shape} \sslash S_n,
\]
with the quotient map $\pi : X_{\shape} \to Y_{\shape}$.
Since the potential $V$ is invariant under both groups, it defines an element of the invariant function field $k(\mathcal{V})^{O(d) \times S_n} \cong k(Y_{\shape})$. We denote the descended potential as the rational map:
\[
    \tilde{f}: Y_{\shape} \dashrightarrow \mathbb{C}.
\]
\end{definition}

The algebraic machinery developed above connects directly to physical symmetry. A critical point's stabilizer in the permutation group encodes the physical point group of the configuration.

\begin{proposition}[Correspondence of Point Groups and Stabilizers]
\label{prop:point_group_iso}
Let $x \in \mathcal{V}(\mathbb{R})$ be a configuration. The physical point group $\mathcal{P}(x) \le O(d,\mathbb{R})$ is isomorphic to the stabilizer of the corresponding shape $q(x)$ under the $S_n$-action:
\[
    \mathcal{P}(x) \cong \operatorname{Stab}_{S_n}(q(x)) := \{\sigma \in S_n \mid q(x) \cdot \sigma = q(x)\}.
\]
\end{proposition}

\begin{proof}
Let $\sigma \in \operatorname{Stab}_{S_n}(q(x))$. By definition of the quotient, $q(x \cdot \sigma) = q(x)$ implies that the permuted configuration $x \cdot \sigma$ lies in the same $O(d, \R)$-orbit as $x$. Thus, there exists a transformation $R \in O(d,\mathbb{R})$ such that $R \cdot x = x \cdot \sigma$. For generic configurations where the orbit map is faithful, the assignment $\sigma \mapsto R$ defines the required isomorphism.
\end{proof}

\subsection{Metric Rarity in the Shape Space Quotient}

Having established the algebraic structure of the total quotient $Y_{\text{shape}}$, we now turn to the metric geometry of the real locus. To quantify the volume of physical configurations within the ambient space, we must first equip the quotient with a measure.

We start by endowing the configuration space $\mathcal{V}$ with a Hermitian metric invariant under both rotations $O(d,\mathbb{R})$ and permutations $S_n$. This structure descends naturally to the quotient, recovering the standard measure of structural similarity used in chemical physics.

\begin{definition}[Root Mean Square Deviation (RMSD) Distance]
\label{def:rmsd_distance}
The distance between two shapes $[x], [y]\in X_{\shape}^+$ in the principal stratum of the physical shape space is defined by the minimal Euclidean distance between their orbits (optimal alignment):
\[
d_X([x],[y]) := \inf_{R \in O(d,\mathbb{R})} \|x - R \cdot y\|_{\mathcal{V}}.
\]
\end{definition}

This distance function arises from a rigorous Riemannian structure.

\begin{proposition}[Riemannian Structure and Distance Compatibility]
\label{prop:metric_standard_result}
There exists a unique smooth Riemannian metric $g_X$ on $X_{\shape}^{\mathrm{\pr}}$ such that the quotient map $q: \mathcal{V}^{{\mathrm{\pr}}} \to X_{\shape}^{{\mathrm{\pr}}}$ 
restricts to a Riemannian submersion.
The intrinsic geodesic distance induced by $g_X$ coincides exactly with the optimal alignment distance $d_X$ defined above.
\end{proposition}
\begin{proof}
The existence of the unique smooth metric $g_X$ follows from the standard construction of quotients by proper, free, and isometric actions (guaranteed by \textbf{Corollary 2.29 in \cite{Lee2018}}).
The fact that the induced Riemannian distance equals the distance between orbits is a property of Riemannian submersions; see \textbf{Proposition 2.109 in \cite{GHL2004}}.
\end{proof}

Since the permutation group $S_n$ acts by isometries on the Riemannian manifold $(X_{\shape}, g_X)$, we can push the metric structure forward to the final quotient. We define the measure on $Y_{\shape}$ explicitly via the orbifold pushforward.

\begin{definition}[Induced Geometric Measure on $Y_{\shape}(\mathbb{R})$]
\label{def:algebraic_involutions_measure_shape}
Let $g_X$ be the Riemannian metric on the principal stratum of $X_{\shape}$ (Proposition \ref{prop:metric_standard_result}). We define the measure $\mu_{Y_{\shape}}$ on the real locus of the total quotient via the pullback to the covering space:
\[
    \mu_{Y_{\shape}}(U) := \frac{1}{|S_n|} \mathcal{H}^D_{(X_{\shape}, g_X)}\left(\pi^{-1}(U)\right) \quad \text{for any measurable } U \subset Y_{\shape}(\mathbb{R}),
\]
where $\mathcal{H}^D$ is the $D$-dimensional Hausdorff measure and $D = \dim_{\mathbb{R}} Y_{\shape}(\mathbb{R})$.
\end{definition}
We are now ready to quantify the relative volume of the real image $L = \pi(X_{\shape}(\mathbb{R}))$ within the ambient quotient space $Y_{\shape}(\mathbb{R})$.
We first require a technical lemma ensures that metric symmetries of the covering space descend to measure-preserving maps on the quotient.

\begin{lemma}[Descending Isometries]
\label{lem:descend_isometry_shape}
Let $\Phi:\mathcal{V}^{\mathrm{pr}}\to \mathcal{V}^{\mathrm{pr}}$ be an isometry that commutes with the $O(d,\mathbb{R})$-action. Then $\Phi$ induces a well-defined isometry $\overline{\Phi}$ on the quotient $X_{\shape}^{\mathrm{pr}}$ with respect to the optimal-alignment metric $d_X$. Consequently, $\overline{\Phi}$ preserves the Hausdorff measure $\mathcal{H}^D$.
\end{lemma}

\begin{proof}
If $[x]=[y]$, then $y=R\cdot x$ for some rotation $R$. By commutativity and isometry of $\Phi$, we have $\Phi(y) = R \cdot \Phi(x)$, so $[\Phi(y)] = [\Phi(x)]$. The fact that $\overline{\Phi}$ is an isometry follows directly from the definition of the distance $d_X$ as an infimum over orbits.
\end{proof}

\begin{theorem}[Relative Volume of $\pi$ real image]
\label{thm:relative_volume_shape_Y}
Let $B(R) \subset Y_{\shape}(\mathbb{R})$ be a metric ball centered at the origin. The relative volume of the physical configurations within the ambient quotient space is given by:
\[
\frac{\mu_{Y_{\shape}}\left(L \cap B(R)\right)}{\mu_{Y_{\shape}}\left(B(R)\right)} = \frac{1}{\#\mathrm{Inv}(S_n)}.
\]
\end{theorem}

\begin{proof}
A point $y \in Y_{\shape}(\mathbb{R})$ is real if and only if its preimage lies in a twisted real sector. Specifically, $\pi^{-1}(Y_{\shape}(\mathbb{R}))$ decomposes (up to measure zero) into a union over involutions:
\[
    \bigcup_{\sigma \in \mathrm{Inv}(S_n)} U_\sigma, \quad \text{where } U_\sigma := \{ x \in X_{\shape} \mid \overline{x} = x \cdot \sigma \}.
\]
The physical image $L$ corresponds solely to the identity sector $U_{\mathrm{id}}$.
We claim all sectors have identical ball volumes. For any involution $\sigma$, there exists a unitary matrix $A_\sigma$ such that $\overline{A_\sigma} = A_\sigma \cdot \sigma$. The map $\Phi_\sigma(v) = v A_\sigma$ is a global isometry of $\mathcal{V}$ that commutes with rotations $O(d)$. By Lemma \ref{lem:descend_isometry_shape}, it descends to an isometry mapping $U_{\mathrm{id}}$ onto $U_\sigma$.
Summing over all sectors yields the result.
\end{proof}
The rarity established above accounts for the permutation structure. However, a second geometric constraint arises from the quotient by the continuous group $O(d,\C)$. The real locus of the shape space contains non-physical configurations corresponding to indefinite metric tensors.

\begin{theorem}[Relative Volume of $q$ real image]
\label{thm:relative_volume_shape_X}
Let $r = \min(d,n)$. For any ball $B(R) \subset X_{\shape}(\mathbb{R})$, the relative volume of the physical configurations $X_{\shape}^+= q(\mathcal{V}(\mathbb{R}))$ is:
\[
\frac{\vol_{X_{\shape}}\left(X_{\shape}^+\cap B(R)\right)}{\vol_{X_{\shape}}\left(B(R)\right)} = \frac{1}{2^r}.
\]
\end{theorem}

\begin{proof}
We analyze the local geometry of the quotient map via the spectral properties of the Gram matrix.

\textbf{Step 1: Spectral Decomposition and Signatures.}
The shape space $X_{\shape}$ is isomorphic to the variety of symmetric $n \times n$ matrices of rank at most $d$. 
Let $r = \min(d,n)$. For a generic real point $x \in X_{\shape}(\mathbb{R})$, 
the corresponding Gram matrix possesses $r$ non-zero real eigenvalues.
The map $q: \mathcal{V} \to X_{\shape}$, given by $q(v) = v^\top v$, maps physical configurations 
$v \in \mathcal{V}(\mathbb{R})$ to positive semi-definite matrices (all eigenvalues positive).
However, the ambient real quotient $X_{\shape}(\mathbb{R})$ generically decomposes into $2^r$ disjoint open 
sectors ("orthants") $\mathcal{O}_\epsilon$, indexed by the signs of the eigenvalues $\epsilon \in \{-1, 1\}^r$. 
The physical image $X_{\shape}^+$ is precisely the totally positive sector.

\textbf{Step 2: Construction of the Isometry.}
We demonstrate that all sectors $\mathcal{O}_\epsilon$ have identical ball volumes.
Fix a signature $\epsilon$ with $p$ positive and $q$ negative signs. 
The preimage $\pi^{-1}(\mathcal{O}_\epsilon)$ consists of configurations where, up to a basis change, $p$ spatial rows are real and $q$ are purely imaginary.
Define the linear map $U: \mathcal{V} \to \mathcal{V}$ that multiplies the spatial coordinates corresponding to the $q$ negative eigenvalues by $i$. This map defines a bijection between the real locus $\mathcal{V}(\mathbb{R})$ (the preimage of the physical sector) and the twisted subspace $\pi^{-1}(\mathcal{O}_\epsilon)$.

\textbf{Step 3: Metric Invariance via Schur's Lemma.}
The transformation $U$ is an isometry of the ambient Hermitian space $(\mathcal{V}, h)$.
The metric $h$ is required to be $O(d)$-invariant. Since the standard representation of $O(d)$ on the spatial coordinates $\mathbb{C}^d$ is irreducible, \textbf{Schur's Lemma} implies that the restriction of $h$ to the spatial factor must be proportional to the standard Euclidean inner product.
Since $U$ acts on the spatial coordinates simply by scalar multiplication by $i$ (a unitary operation, as $|i|=1$), it preserves the standard inner product. Consequently, $U$ preserves the global metric $h$.

\textbf{Conclusion.}
As a global isometry of the covering space, $U$ preserves the Hausdorff measure of the preimages:
\[
\mathcal{H}^D(\pi^{-1}(X_{\shape}^+ \cap B(R))) = \mathcal{H}^D(\pi^{-1}(\mathcal{O}_\epsilon \cap B(R))).
\]
By the definition of the quotient measure (Definition \ref{def:algebraic_involutions_measure_shape}), equal preimage volumes imply equal quotient volumes. Summing over the $2^r$ distinct signatures yields the relative volume factor $1/2^r$.
\end{proof}
We now combine the topological constraints (involutions) and the geometric constraints (signature) to quantify the total rarity of the physical domain.

\begin{theorem}[Relative Volume of the Real Image]
\label{thm:relative_volume_real_image_shape}
For any metric ball $B(R) \subset Y_{\shape}(\mathbb{R})$ centered at the origin, the relative volume of the physical image $L = \pi(X_{\shape}^+)$ is given by:
\[
    \frac{\mu_{Y_{\shape}}(L \cap B(R))}{\mu_{Y_{\shape}}(B(R))} = \frac{1}{2^{\min(d,n)} \cdot \#\mathrm{Inv}(S_n)}.
\]
\end{theorem}

\begin{proof}
The physical image is the intersection of two conditions: lying in the trivial involution sector ($\sigma=\mathrm{id}$) and having positive semi-definite signature.
Since the action of $S_n$ on the shape space is by conjugation of the Gram matrix, it preserves the eigenvalues. Thus, the spectral condition (signature) is independent of the real-structure condition (involutions). The total relative volume is the product of the individual factors derived in Theorems \ref{thm:relative_volume_shape_Y} and \ref{thm:relative_volume_shape_X}.
\end{proof}

\begin{corollary}[Asymptotic Decay]
\label{cor:Inv_Sn_decay_shape}
Since the number of involutions in the symmetric group grows as $\#\mathrm{Inv}(S_n) \sim n^{n/2}$, the relative volume of the physical configurations decays super-exponentially with the system size:
\[
    \frac{\mu_{Y_{\shape}}\left(L \cap B(R)\right)}{\mu_{Y_{\shape}}\left(B(R)\right)} \approx O\left(e^{-C n \log n}\right).
\]
\end{corollary}

\begin{remark}[$\tilde{f}_{LJ}$ is Unbounded Below on $Y_{\shape}(\mathbb{R})$]
\label{rem:unbounded_dive}
We explicitly construct a collision-free trajectory that drives the quotient Lennard-Jones potential $\tilde{f}_{LJ}$ to $-\infty$.
Consider a 3-particle system $\mathcal{V} = \mathbb{C}^{3 \times 3}$ we define the twisted real space 
\[  
U_{(12),-I} = \{ x \in \mathcal{V} \mid \overline{x} = - x \cdot (12) \}.
\]
Points in this space have the form:
\[
x_1 = z, \quad x_2 = -\overline{z}, \quad x_3 = 0
\]
Let $z = u + i v$ with $u, v \in \mathbb{R}^3$. 
Then note that $\overline{q\circ \pi (x)} = q\circ \pi (\overline{x}) = q\circ \pi (-x) = q(\pi (x)\cdot(12)) = q\circ \pi (x)$, confirming that $q\circ \pi (U_{(12),-I})\subset Y_{\shape}(\mathbb{R})$. 
This symmetry dictates the squared distances:
Let $r_{ij}^2 = \langle x_i - x_j, x_i - x_j \rangle$. Then: 
\[
r_{13}^2  = \langle z, z \rangle, \quad r_{23}^2 = \langle \overline{z}, \overline{z} \rangle = \overline{r_{13}^2}, \quad r_{12}^2 = 4|u|^2.
\]

Let $u(\rho) = \sqrt{\rho} \cos(\pi/12) \mathbf{e}_1$ and $v(\rho) = \sqrt{\rho} \sin(\pi/12) \mathbf{e}_1$.
Then $x(\rho) = (u(\rho) + i v(\rho), -u(\rho) + i v(\rho), 0)$ 
Then note that:
\[
r_{13}^{-12} = (\rho e^{i\pi/6})^{-6} = \rho^{-6} e^{-i \pi} = -\frac{1}{\rho^6}.
\]
And, 
\[ 
r_{12}^{-12} = (4|u|^2)^{-6} = (4 \rho \cos^2(\pi/12))^{-6} = \frac{C}{\rho^6} \quad \text{where } C = (4 \cos^2(\pi/12))^{-6} < 2.
\]
Therefore, 
\begin{align*}
\text{lim}_{\rho \to 0}V_{LJ}(x(\rho)) &= \text{lim}_{\rho \to 0} \left( r_{13}^{-12} - r_{13}^{-6} + r_{23}^{-12} - r_{23}^{-6} + r_{12}^{-12} - r_{12}^{-6} \right) \\
&= \text{lim}_{\rho \to 0} \left( -\frac{2-C}{\rho^6}  \right) = -\infty.
\end{align*}
\end{remark}

\begin{remark}[Contextualization via Galois Cohomology]
\label{rem:galois_cohomology}
The volume formula derived in Theorem \ref{thm:relative_volume_real_image_shape} for $G = O(d,\mathbb{C}) \times S_n$ can be interpreted as a specific instance of a broader algebraic phenomenon.
For a general reductive group $G$, the real locus of the quotient $(V \sslash G)(\mathbb{R})$ decomposes into disjoint sectors indexed by the Galois cohomology set $H^1(\mathbb{R}, G)$.
The "physical" image $L$ corresponds to the trivial cohomology class, while the complementary sectors arise from non-trivial real forms (twists) of the representation.

The metric weight of these sectors depends on the geometry of the defining cocycles.
With respect to a natural invariant measure, the volume of a sector $Y_\alpha$ is proportional to the volume of the conjugacy class associated with $\alpha$ within the maximal compact subgroup $K \subset G$.
Our result illustrates two distinct regimes of this principle:
\begin{enumerate}
    \item \textbf{Geometric Equipartition ($O(d)$):} The factor $2^{-\min(d,n)}$ arises because the relevant twisted conjugacy classes in the compact form are geometrically equivalent (via unitary transformations), leading to equal volumes for all signatures.
    \item \textbf{Combinatorial Weighting ($S_n$):} The factor $(\#\mathrm{Inv}(S_n))^{-1}$ arises because the "volume" corresponds to the cardinality of conjugacy classes. Here, the physical sector (the identity) has size 1, which is negligible compared to the large classes of involutions.
\end{enumerate}
This framework suggests that "metric rarity" is to be expected in general reductive groups whenever the cohomology set is large and the trivial class does not dominate the measure in the compact form.
\end{remark}

\section{Unified Algebraic Interpretation of Empirical Symmetry}
\label{sec:empirical_implications}

The prevalent symmetry observed in \cite{schneider2025} and \cite{arjevani2019, arjevani2021} provides a comprehensive empirical validation of ``Regime I''. In these experimental settings, the optimization landscape is governed by the metric rarity of the real image $L \subset Y(\mathbb{R})$. As established in Section \ref{subsec:metric_rarity}, the ``interior'' of the quotient space (corresponding to asymmetric configurations) occupies negligible volume. Consequently, generic critical points are statistically constrained to the boundary strata $\partial L$, necessitating non-trivial stabilizers.

To demonstrate the universality of this mechanism, we map the diverse experimental environments of \cite{schneider2025} 
and the deep learning landscapes studied by Arjevani et al. directly to our linear $G$-variety 
framework $(X, G)$ in Table \ref{tab:mapping}.

\begin{table}[h]
\centering
\caption{Mapping Empirical Setups to the $G$-Variety Framework}
\label{tab:mapping}
\begin{tabular}{@{}lll@{}}
\toprule
\textbf{Experiment} & \textbf{Configuration Space} $X(\mathbb{R})$ & \textbf{Acting Group} $G$ \\ \midrule
Projective Optimization \cite{schneider2025} & Functions $\mathbb{R}^{\mathbb{P}^n(\mathbb{F}_q)}$ & $PGL_{n+1}(\mathbb{F}_q)$  \\
Octahedral Graph \cite{schneider2025} & Vertex Valuations $\mathbb{R}^{6}$ & $\text{Aut}(\Gamma) \cong S_4 \times C_2$  \\
Perfect Matching \cite{schneider2025} & Vertex Valuations $\mathbb{R}^{|V|}$ & $\text{Aut}(\Gamma)$  \\
Particle Systems ($I_V$) \cite{schneider2025} & Configurations $(\mathbb{R}^d)^n$ & $S_n$  \\ 
Neural Nets / Tensors \cite{arjevani2019, arjevani2021} & Weight Matrices $\mathbb{R}^{k\times d}$ & $S_k \times S_d$ (Row/Col Perm.) \\ \bottomrule
\end{tabular}
\end{table}

A central finding in \cite{schneider2025} is the persistence of symmetry even when the standard stabilizer $I_V$ is trivial (e.g., in repulsive potentials). This phenomenon, quantified by the \textbf{Edge Isotropy} $I_E$, admits a geometric interpretation as the stabilizer within the \textit{Shape Space}.

While $I_V(x)$ is the stabilizer of the configuration $x \in \mathcal{V} = (\mathbb{C}^d)^n$ under the linear action of $S_n$, the potential depends only on intrinsic distances, descending to the quotient variety $X_{\shape} = \mathcal{V} \sslash O(d,\mathbb{C})$. We identify the Edge Isotropy group as the stabilizer of the \emph{image} of $x$ in this intermediate quotient:
\[
I_V(x) = \text{Stab}_{S_n}(x) \quad \subseteq \quad I_E(x) \cong \text{Stab}_{S_n}(q(x)),
\]
where $q: \mathcal{V} \to X_{\shape}$ is the quotient map by global motions.

The empirical observation that $I_E$ remains non-trivial even when $I_V$ vanishes confirms our geometric rarity hypothesis. Even if a potential (e.g., $1/r^2$) is singular at the collision locus—preventing critical points on the boundary of $L\subset (\mathcal{V}/S_n)(\mathbb{R})$ (where $I_V \neq \{e\}$)—the system settles on the boundary of the shape space $L\subset Y_{\shape}(\mathbb{R})$. The metric rarity of the real image in the final quotient $Y_{\shape} = X_{\shape} \sslash S_n$ (Theorem \ref{thm:relative_volume_real_image_shape}) is sufficient to statistically favor symmetric shapes with non trivial symmetry ($I_E \neq \{e\}$). Moreover, the additional volume factor of $1/2^{\min(d,n)}$ in Theorem \ref{thm:relative_volume_real_image_shape} explains the consistent empirical strict inclusion $I_V(x) < I_E(x)$ reported in \cite{schneider2025}.

\section{Conclusion and Future Work}
\label{sec:conclusion}
Throughout this work, we have adopted the following general setting: an irreducible complex affine algebraic variety $X$ defined over $\mathbb{R}$, equipped with a faithful action of a finite group $G$ by real automorphisms, and a $G$-invariant regular function $f\in \mathbb{R}[X]^G$ viewed as an energy landscape on the real locus $X(\mathbb{R})$. 
The central problem addressed herein is to determine conditions under which critical points---and in particular global minimizers---exhibit a non-trivial stabilizer. 
Across the families of models considered here, we observed two recurring empirical regimes: in Regime~I almost all critical 
points exhibit a non-trivial stabilizer, while in Regime~II, a distinct energetic ordering emerges: high-symmetry configurations are systematically concentrated at the bottom of the landscape, with the deepest minima exhibiting larger stabilizer subgroups than the bulk.

Instead of working directly on the configuration space $X(\mathbb{R})$, we study the descended potential $\tilde f$ on the quotient $Y=X\sslash G$ and its restriction to the \emph{physical real image}
\[
L=\pi\big(X(\mathbb{R})\big)\subset Y(\mathbb{R}).
\]
Our central structural point is that $L$ is typically \emph{metrically rare} inside the ambient real quotient. 
This rarity has a precise algebro-geometric origin: the preimage of the real locus $Y(\mathbb{R})$ decomposes into a union of real forms indexed by involutions,
\[
\pi^{-1}\big(Y(\mathbb{R})\big)\cong \bigcup_{\sigma\in\mathrm{Inv}(G)} X_\sigma(\mathbb{R}),
\]
whereas the physically accessible component corresponds only to the image of the \emph{trivial} real structure. 
As the number of involutions grows, the relative size of the physical image shrinks. 

This geometric constraint offers a unified explanation for Regime I. When $L$ is metrically rare, the critical points of the descended function $\tilde{f}$ are statistically unlikely to lie within the smooth locus of the physical domain $L$. Consequently, the critical points of $f$ arise strictly from the singularities of the quotient map $\pi$ (where the differential $d\pi$ is not surjective). We showed that these singularities coincide exactly with the locus of symmetric configurations in $X$; thus, in this regime, critical points are forced to be symmetric.

For \textbf{Regime II}, we proposed the ``Active Constraint'' hypothesis. We posit that even when the descended potential $\tilde{f}$ is sufficiently complex to exhibit a multimodal landscape within the interior $L^\circ$, the global minimum appears to be driven to the boundary $\partial L$ by a coherent global gradient of $\tilde{f}$ that does not vanish in the interior.
We found strong empirical support for this mechanism in Lennard-Jones clusters, where the well-documented ``funnel'' topology effectively directs the system toward a \textbf{highly symmetric global minimum}. In our geometric interpretation, this implies that the global gradient on the quotient space channels the trajectory specifically into singular strata of high codimension of the boundary $\partial L$.

\paragraph{Limitations.}
The interpretation of Regime~I depends on identifying the correct \emph{measure} governing critical points; our estimates use the metric descended from the Hermitian metric on $X$ as a proxy, but a stratified Kac--Rice theory for a given random model on $L$ would yield a principled density. 

\paragraph{Future directions.}
We envision four primary avenues for extending this geometric framework:

\begin{itemize}
  \item \textbf{Rigorous Formulation of Global Gradients.}
    While we observed empirically that random $S_n$ invariant polynomials exhibit higher average symmetry at energetic extremes, a theoretical explanation for this phenomenon remains to be established. 
    Future work is required to rigorously formulate the definition of a ``global gradient'' within the context of random algebraic landscapes, and to verify whether such gradients naturally emerge in generic ensembles restricted to the real image $L$. 
    Confirming this would identify the global gradient as the causal factor driving the statistical preference for high symmetry at the boundaries.

    \item \textbf{The Ambient Landscape Topology.} A natural question is to what extent the ``funnel'' topology characteristic of Lennard-Jones clusters is an intrinsic feature of the algebraic potential $\tilde{f}$ on the entire ambient space $Y(\mathbb{R})$.
We acknowledge that since $\tilde{f}$ is unbounded below on the ambient quotient (as shown in Remark \ref{rem:unbounded_dive}), formulating this problem rigorously requires imposing appropriate bounds or regularization.
Nevertheless, the fundamental question remains: if one were to analyze the landscape on this extended domain, ignoring the specific geometric constraints of the real image, would the funnel structure persist? Or would the landscape degenerate into a ``glassy'' structure characterized by frustration and a large number of competing minima?

\item \textbf{Crystallization and Stratum Dominance.} 
    Our geometric identification of symmetry with the boundary of the real image $L \subset Y(\mathbb{R})$ offers a novel perspective on the symmetry selection observed in crystallization.
    Rather than inquiring which symmetric structures possess the intrinsic physical stability to host a global minimum, we propose a geometric inversion of the problem: we ask instead which \emph{boundary wall} of the domain $L$ is topologically positioned to intercept the "funnel" of the ambient potential $\tilde{f}$.
    In this view, the specific point-group selection is dictated by the interaction between the global gradient drift of the potential and the stratification of $\partial L$.
    For instance, Theorem~\ref{thm:boundary_parity} establishes that strata corresponding to odd-order stabilizers map to the interior $L^\circ$. 
    We propose, as a \textbf{first-order heuristic}, that this topological placement provides the geometric rationale for their exclusion: under the Active Constraint mechanism—where the global gradient drives the system outwards—we expect these symmetries to be disfavored simply because they do not constitute a "boundary wall" capable of arresting the gradient's descent.
    Likewise, we suggest formulating criteria to predict which boundary strata are \emph{dominant}, thereby providing a geometric derivation for the specific symmetries observed in crystallized matter.

  \item \textbf{Beyond Group Actions: General Morphisms and Deep Learning.} 
      The prevalence of symmetry in neural networks has been rigorously documented in what we term \textbf{Regime I}. Arjevani et al.\ \cite{arjevani2019} showed that loss functions of shallow ReLU networks systematically preserve the permutation symmetries of hidden neurons, rendering symmetry prevalent. 
      Empirical analyses of machine learning energy landscapes \cite{ballard2017energy} reveal a ``funnel-like'' topography analogous to the single-funnel landscapes of physical systems like Lennard-Jones clusters, which facilitates efficient relaxation to low-energy states.
      To provide a unified geometric explanation for these phenomena in Deep Learning, we propose looking beyond group quotients.
      The phenomenon of real-image compression is not specific to quotients; for a general subjective finite morphism $\phi:X\to Y$ defined over $\mathbb{R}$, we expect the 2-primary part of the degree $\deg(\phi)$ to control the extent to which the real image $\phi(X(\mathbb{R}))$ is ``squeezed'' inside $Y(\mathbb{R})$. 
      This suggests a speculative geometric lens for analyzing neural network loss landscapes: if a loss function on weight space factors (even approximately) as $f=\tilde f\circ \phi$, where $\phi$ quotients out built-in symmetries and structural dependencies, then optimization is effectively constrained to the real image $L_\phi:=\phi(X(\mathbb{R}))$. 
      Successive ``folds'' of $L_\phi$ could then offer a geometric route to ``deep singularities,'' constraining the feasible real locus and pushing global minima toward increasingly singular corners, in direct analogy with the Lennard-Jones phenomenon studied here.
\end{itemize}

\section*{Acknowledgments}

I would like to express my deepest gratitude to David Kazhdan for his unwavering mentorship and guidance throughout this research. His constant availability, his patience in answering countless questions, and his readiness to discuss any ideas were invaluable. This work would not have been possible without his support.

I am deeply indebted to Yossi Arjevani, my former advisor, for originally introducing me to the fundamental questions regarding symmetry and for the numerous insightful discussions we shared. His guidance was fundamental to the initiation of this research.

I also thank Shmuel Weinberger, Eliran Subag, and David Wales for many enlightening discussions and valuable insights.

I am especially grateful to Will Sawin for his significant contribution to the geometric analysis in Section 4, and specifically for providing the precise formulation and proof of Theorem \ref{thm:rarity_real_image_involutions}.
% ============================================================
% Appendix: PSC-based counting calibration for Regime I
% (Insert this near the end of the paper, before the bibliography)
% ============================================================

\appendix

\section{A Principle of Symmetric Criticality-based counting}
\label{app:capacity}

This appendix records a simple counting baseline for the prevalence of symmetric critical points.
Its role is \emph{calibrational}: it explains why, in certain parameter ranges (notably moderate degree and/or strong symmetry),
one should \emph{already expect} many symmetric critical points purely from the combination of
(i) the Principle of Symmetric Criticality (Theorem~\ref{thm:PSC}) and
(ii) crude estimates for the expected number of real solutions of polynomial systems (the ``square-root'' scaling, e.g.\ \cite{EdelmanKostlan1995}).
We emphasize from the outset that this perspective concerns \emph{counts} rather than \emph{energetic ordering};
in particular it is silent about Regime~II (concentration of \emph{higher} symmetry at energetic extremes).
Moreover, even for Regime~I it does not explain the \emph{mechanism}: a counting balance may predict that few (or no) asymmetric critical points
should exist in a given parameter regime, but it does not explain \emph{how} an invariant polynomial ``avoids'' the open dense asymmetric locus.
The quotient-geometric mechanism developed in the main text addresses this missing ``how'' and is essential for the heuristic explanation of Regime~II.

\subsection*{Setup and heuristic}
Let $k,n\in\N$ and set
\[
X(\R) := (\R^k)^n \cong \R^{kn}.
\]
We write $x=(x_1,\dots,x_n)$ with $x_i\in\R^k$.
Let $G:=S_n$ act on $X(\R)$ by permuting labels:
\[
(\sigma\cdot x)_i := x_{\sigma(i)},\qquad \sigma\in S_n.
\]
Let $f:X(\R)\to\R$ be an $G$-invariant polynomial of degree $d$.

For any transposition $(ij)\in S_n$, the fixed-point locus
\[
\Delta_{ij}:=\{x\in X(\R)\mid x_i=x_j\}
\]
is a symmetric stratum (a special case of Definition~\ref{def:fixed_point_space}).
By the Principle of Symmetric Criticality (Theorem~\ref{thm:PSC}), any critical point of $f|_{\Delta_{ij}}$ is a critical point of $f$ on all of $X(\R)$.

We now invoke a standard heuristic from random real algebraic geometry:
for a ``generic'' polynomial system in $m$ real variables with degrees comparable to $D$ (e.g.\ Kostlan/KSS-type ensembles),
the expected number of \emph{real} solutions scales like $D^{m/2}$ (the ``square-root'' scaling; cf.\ \cite{EdelmanKostlan1995}).
Applied to the gradient system, this suggests
\[
\#\text{Crit}_\R(f)\ \asymp\ (d-1)^{kn/2}.
\]
Restricting to $\Delta_{ij}$ reduces the dimension by $k$, hence
\[
\#\text{Crit}_\R(f|_{\Delta_{ij}})\ \asymp\ (d-1)^{(kn-k)/2}=(d-1)^{k(n-1)/2}.
\]
Summing over the $\binom{n}{2}$ pair-collision strata yields the leading-order symmetric contribution
\[
\#\text{Crit}_\R^{\mathrm{sym}}(f)\ \asymp\ \binom{n}{2}(d-1)^{k(n-1)/2},
\]
and therefore the symmetric-to-total ratio satisfies the heuristic scaling
\begin{equation}
\label{eq:capacity_ratio}
\frac{\#\text{Crit}_\R^{\mathrm{sym}}(f)}{\#\text{Crit}_\R(f)}
\ \asymp\
\frac{\binom{n}{2}}{(d-1)^{k/2}}
\ =\
\frac{n(n-1)}{2(d-1)^{k/2}}.
\end{equation}
In particular, for fixed $n$, this ratio tends to $0$ as either $d\to\infty$ or $k\to\infty$.

\subsection*{Consequence for $k=1$: a quartic-in-$n$ degree threshold}
In the one-dimensional model $X(\R)=\R^n$ ($k=1$), \eqref{eq:capacity_ratio} becomes
\[
\frac{\#\text{Crit}_\R^{\mathrm{sym}}(f)}{\#\text{Crit}_\R(f)}
\ \asymp\ \frac{n(n-1)}{2\sqrt{d-1}}.
\]
A natural benchmark for observing a substantial asymmetric population is that symmetric and asymmetric critical points are comparable,
i.e.\ the symmetric fraction is on the order of $1/2$. Setting the right-hand side to $1/2$ gives
\[
\frac{n(n-1)}{2\sqrt{d-1}} \approx \frac12
\quad\Longleftrightarrow\quad
\sqrt{d-1}\approx n(n-1)
\quad\Longleftrightarrow\quad
d \approx 1+n^2(n-1)^2.
\]
Thus, to enter an ``$\sim 50\%$ asymmetric'' regime one requires
\[
d=\Theta(n^4),
\]
which is already very large for moderate $n$.

\subsection*{Empirical calibration: $(k,n,d)=(2,4,14)$}
For $k=2$ and $n=4$, \eqref{eq:capacity_ratio} yields
\[
\frac{\#\text{Crit}_\R^{\mathrm{sym}}(f)}{\#\text{Crit}_\R(f)}
\ \asymp\ \frac{\binom{4}{2}}{(d-1)}
\ =\ \frac{6}{d-1}.
\]
At $d=14$ this predicts a symmetric fraction of $\approx 6/13\approx 0.46$ (hence an asymmetric fraction of $\approx 0.54$).

To compare this heuristic with a simple numerical benchmark, we ran the following experiment.
We repeated $100$ independent runs; in each run we sampled a fresh degree-$14$ $S_4$-invariant polynomial and applied a damped Newton method from a random initialization, recording the single converged critical point returned by the solver. We then classified this point as asymmetric (trivial stabilizer) or symmetric. The returned critical point was asymmetric in $42\%$ of runs, broadly consistent with the crossover-scale prediction.

\section{Additional Experiments: Robustness of the Regime II Phenomenon}
\label{app:boundary_experiments}

This appendix presents additional numerical experiments that corroborate the existence of \textbf{Regime II}—the energetic preference for higher symmetry at global minima—within random $S_n$-invariant polynomials (as illustrated in Figure~\ref{fig:distinct_values}).

\subsection*{Methodology and Sampling}
To quantify symmetry, we utilize the same score defined in the main text: for a critical point $x\in\R^n$, we record the \emph{number of distinct coordinate values} (where fewer distinct values imply a larger stabilizer and higher symmetry).
For each sampled polynomial, we detect critical points, sort them by energy, and normalize their index to $t\in[0,1]$. The reported profiles are averages over the sampled ensemble.

\paragraph{Numerical tolerances (convergence and symmetry).}
Unless stated otherwise, throughout Appendix~\ref{app:boundary_experiments} (and in the main Figure~\ref{fig:distinct_values}) we use the following fixed tolerances.
First, we declare that a Newton trajectory has converged to a (numerical) critical point when the gradient norm falls below
\[
\|\nabla f(x)\| < \varepsilon,
\qquad \varepsilon = 10^{-1}.
\]
For constrained runs on a manifold (e.g.\ $\|x\|=1$ or $\|e(x)\|=1$), we apply the same rule to the \emph{tangent-projected} gradient, i.e.\ $\|\nabla_T f(x)\|<\varepsilon$.
Second, to compute the symmetry score $D(x)$, we treat two coordinates as equal when $|x_i-x_j|\le \delta$ with
\[
\delta = 10^{-2}.
\]
Operationally, we sort the coordinates of $x$ and set $D(x)=1+\#\{i:\,x_{(i+1)}-x_{(i)}>\delta\}$.
The same tolerance $\delta$ is also used when deduplicating numerically detected critical points up to permutation symmetry.

All experiments utilize the \textbf{non-homogeneous KSS ensemble} (Definition~\ref{def:nhkss}).

\begin{definition}[Non-homogeneous KSS ensemble]
\label{def:nhkss}
Fix $m,d\in\N$ and write $z=(z_1,\dots,z_m)$. The (non-homogeneous) Kostlan--Shub--Smale (KSS) ensemble consists of random polynomials
\[
R(z)=\sum_{|\alpha|\le d} \xi_\alpha\, \sqrt{\binom{d}{\alpha}}\, z^\alpha,
\qquad \xi_\alpha\sim\mathcal{N}(0,1)\ \text{i.i.d.},
\]
where for a multi-index $\alpha=(\alpha_1,\dots,\alpha_m)\in\N^m$ we set $z^\alpha=\prod_{i=1}^m z_i^{\alpha_i}$, $|\alpha|=\sum_{i=1}^m\alpha_i$, and
\[
\binom{d}{\alpha}:=\frac{d!}{\alpha_1!\cdots\alpha_m!\,(d-|\alpha|)!}.
\]
\end{definition}

\subsection*{Symmetrization Schemes}
We investigate two distinct constructions to impose $S_n$-invariance on the random landscape.

\begin{definition}[Reynolds symmetrization]
\label{def:reynolds_symmetrization}
Given a polynomial $Q:\R^n\to\R$, its Reynolds symmetrization (with respect to the permutation action of $S_n$) is
\[
f(x)=\frac{1}{n!}\sum_{\sigma\in S_n} Q(\sigma\cdot x).
\]
In our experiments, $Q$ is sampled from the non-homogeneous KSS ensemble (Definition~\ref{def:nhkss}) with $m=n$.
\end{definition}

\begin{definition}[Quotient-KSS model (quotient composition)]
\label{def:quotient_kss_model}
Let $\pi(x)=(e_1(x),\dots,e_n(x))$ be the elementary-symmetric quotient map $\pi:\R^n\to\R^n$. Sample a polynomial $P:\R^n\to\R$ from the non-homogeneous KSS ensemble (Definition~\ref{def:nhkss}) on the quotient variables $y$, and define
\[
f(x)=P\big(\pi(x)\big)=\tilde f\big(\pi(x)\big).
\]
\end{definition}

\subsection*{Robustness Checks: Coercivity and Domain Constraints}
To ensure that the preference for symmetry is a structural geometric feature rather than an artifact of the polynomial definition, we introduced three types of "blocking" or constraints.

\subsubsection*{Type A: Physical Coercivity (Configuration Space Bounding)}
To mimic physical distance-based potentials (e.g., Lennard-Jones), which are bounded below, we add a coercive term in the configuration space. The modified potential takes the form $f(x) + c\|x\|^{2k}$, where $2k$ is the smallest even integer exceeding $d$. This ensures that the global minimum is not at infinity and resembles physical energy landscapes.

\subsubsection*{Type B: Quotient Space Bounding (Isolating Boundary Bias)}
A potential concern is that $\tilde{f}$ is generally unbounded on the full quotient space $Y(\R)$ (see Remark \ref{rem:unbounded_dive}). One might suspect that the observed boundary bias is merely driven by the system "escaping" to $-\infty$ in $Y(\R)$.
To isolate the tendency towards the boundary from the depth of the potential, we impose the lower bound directly on $\tilde f$ in the quotient space:
\[
\tilde{f}_{bounded}(y) = \tilde{f}(y) + c\|y\|^{2k}.
\]
As shown in Figure~\ref{fig:boundary_appendix_quotient} and Table~\ref{tab:boundary_experiments_summary}, the low-energy tail remains systematically more symmetric even when the potential is coercive on $Y(\R)$, ruling out the "escape to infinity" hypothesis.

\subsubsection*{Type C: Spherical Constraint (Isolating Radial Preference)}
To test whether the boundary effect is driven purely by the radial distance in the quotient, we performed experiments constrained to the ES-sphere:
\[
\|e(x)\|^2=\sum_{k=1}^n e_k(x)^2 = 1.
\]
This effectively "blocks" the radial direction, isolating the angular distribution of critical points.
\begin{itemize}
    \item For $n=2$, the constraint defines a circle in $Y(\R)\cong\R^2$. The distinct-values profile becomes nearly flat, the boundary effect vanishes.
    \item For $n=3$, the constraint defines a sphere in $Y(\R)\cong\R^3$. Here, the low-energy tail \emph{still} shows increased symmetry (see Figure~\ref{fig:boundary_appendix_quotient_sphere}), confirming that the geometric preference for singular strata persists even when radial degrees of freedom are removed.
\end{itemize}

\begin{table}[!htbp]
\centering
\caption{A compact summary of boundary preference. For a critical point $x\in\R^n$, let $D(x)$ be the number of distinct coordinate values. ``Dist.'' reports the empirical distribution of $D(x)$ across all detected points (aggregated over polynomials). ``Avg'' is the mean distinct-values score averaged over all $t$ (per polynomial). The columns $t=0$, $t=0.5$, and $t=1$ report the averages of $D(x)$ at the lowest-energy, median-index, and highest-energy detected critical points within each polynomial. Smaller values indicate higher symmetry.}
\label{tab:boundary_experiments_summary}
{\small
\begin{tabular}{lll p{4.8cm} rrrr}
\toprule
Construction & $(n,d)$ & Coercive term & Dist.\ $D(x)$ & Avg & $t=0$ & $t=0.5$ & $t=1$ \\
\midrule
Reynolds & $(2,10)$ & $\|x\|^{12}$         & {\scriptsize 1: 16.30\%, 2: 83.70\%} & 1.74 & 1.58 & 1.75 & 1.63 \\
Reynolds & $(2,10)$ & $1000\|x\|^{12}$     & {\scriptsize 1: 19.44\%, 2: 80.56\%} & 1.61 & 1.36 & 1.69 & 1.62 \\
Reynolds & $(3,3)$  & $\|x\|^{4}$          & {\scriptsize 1: 5.50\%, 2: 31.41\%, 3: 63.09\%} & 2.13 & 1.37 & 2.23 & 2.02 \\
Reynolds & $(4,4)$  & $\|x\|^{6}$          & {\scriptsize 1: 3.56\%, 2: 17.99\%, 3: 37.09\%, 4: 41.36\%} & 2.53 & 1.47 & 2.69 & 2.30 \\
Reynolds & $(5,3)$  & $\|x\|^{4}$          & {\scriptsize 1: 1.71\%, 2: 7.04\%, 3: 18.17\%, 4: 33.70\%, 5: 39.39\%} & 2.78 & 1.48 & 2.88 & 2.87 \\
Quotient & $(2,10)$ & none                 & {\scriptsize 1: 12.52\%, 2: 87.48\%} & 1.82 & 1.30 & 1.90 & 1.31 \\
Quotient & $(2,10)$ & $\|y\|^{12}$         & {\scriptsize 1: 8.74\%, 2: 91.26\%} & 1.88 & 1.89 & 1.92 & 1.79 \\
Quotient & $(2,20)$ & none                 & {\scriptsize 1: 11.88\%, 2: 88.12\%} & 1.84 & 1.48 & 1.91 & 1.47 \\
Quotient & $(3,3)$  & $\|y\|^{4}$          & {\scriptsize 1: 2.58\%, 2: 28.87\%, 3: 68.55\%} & 2.53 & 2.10 & 2.59 & 2.14 \\
Quotient & $(3,5)$  & $\|y\|^{6}$          & {\scriptsize 1: 3.01\%, 2: 31.89\%, 3: 65.10\%} & 2.51 & 2.30 & 2.58 & 2.19 \\
Quotient & $(3,10)$ & none                 & {\scriptsize 1: 5.05\%, 2: 35.28\%, 3: 59.67\%} & 2.46 & 1.70 & 2.61 & 1.68 \\
Quotient & $(2,8)$  & none ($\|e\|=1$)      & {\scriptsize 1: 9.34\%, 2: 90.66\%} & 1.88 & 1.88 & 1.88 & 1.88 \\
Quotient & $(3,8)$  & none ($\|e\|=1$)      & {\scriptsize 1: 1.36\%, 2: 25.49\%, 3: 73.14\%} & 2.65 & 2.39 & 2.67 & 2.45 \\
Quotient & $(4,4)$  & none                 & {\scriptsize 1: 2.72\%, 2: 19.25\%, 3: 40.86\%, 4: 37.17\%} & 2.97 & 1.69 & 3.18 & 1.68 \\
Quotient & $(4,4)$  & $\|y\|^{6}$          & {\scriptsize 1: 1.75\%, 2: 19.77\%, 3: 42.66\%, 4: 35.83\%} & 2.97 & 2.28 & 3.09 & 2.38 \\
Quotient & $(6,4)$  & $\|y\|^{6}$          & {\scriptsize 1: 1.35\%, 2: 14.08\%, 3: 29.18\%, 4: 29.71\%, 5: 18.86\%, 6: 6.83\%} & 3.43 & 2.34 & 3.55 & 2.59 \\
\bottomrule
\end{tabular}
}
\end{table}
\FloatBarrier

\begin{figure}[p]
\centering
\begin{subfigure}[t]{0.48\linewidth}
\centering
\includegraphics[width=\linewidth]{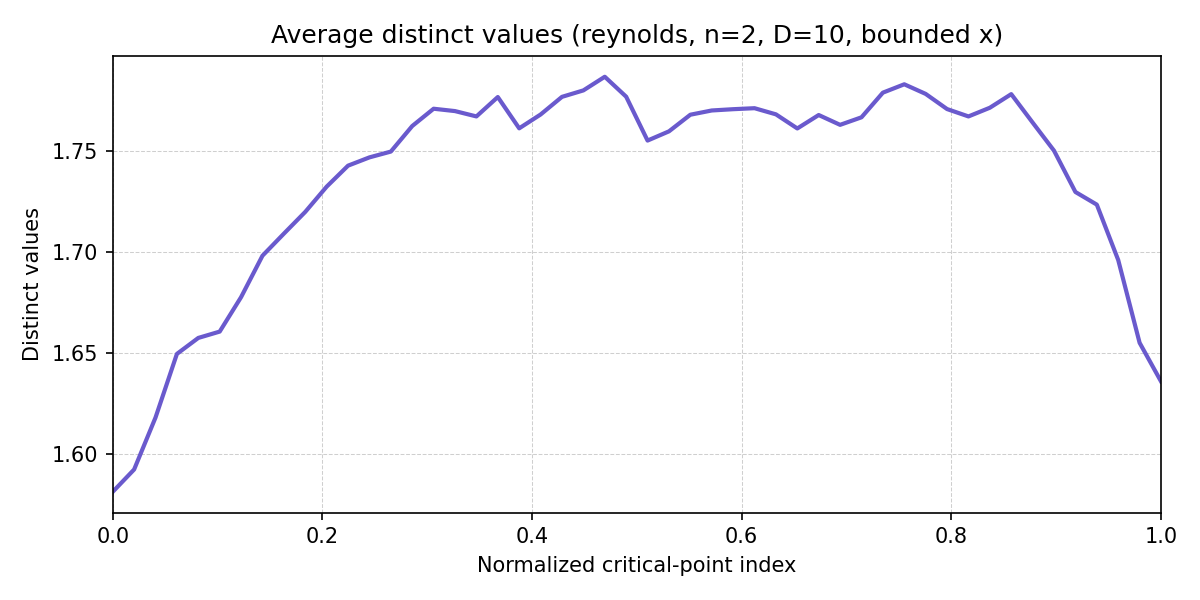}
\caption{Reynolds, $n=2$, $d=10$ (bounded, $+\|x\|^{12}$).}
\end{subfigure}
\hfill
\begin{subfigure}[t]{0.48\linewidth}
\centering
\includegraphics[width=\linewidth]{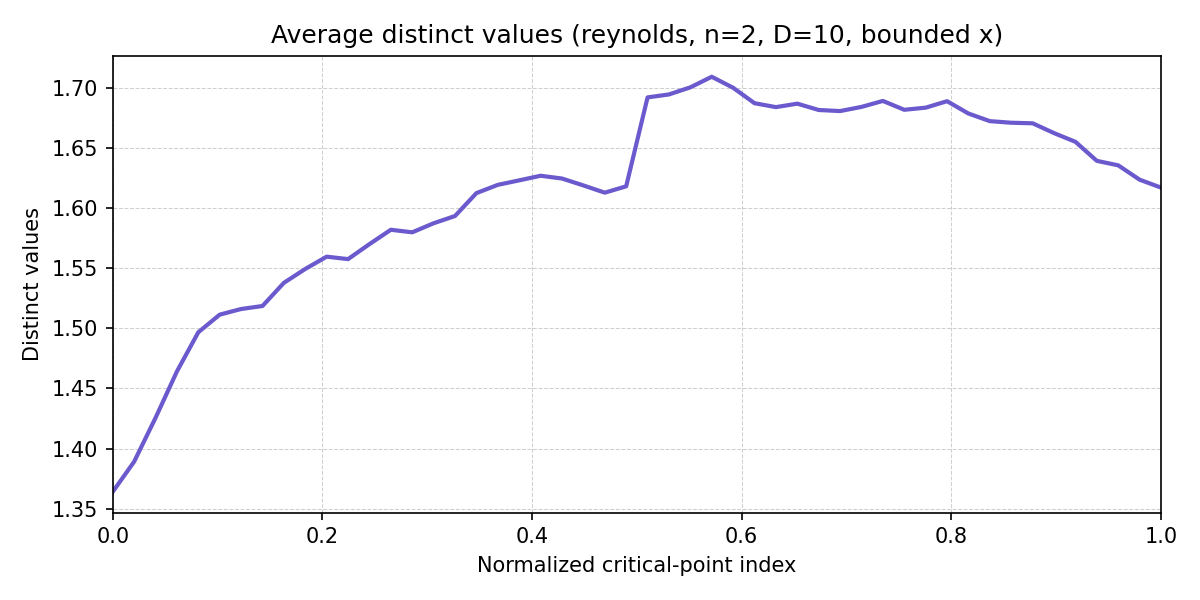}
\caption{Reynolds, $n=2$, $d=10$ (bounded, $+1000\|x\|^{12}$).}
\end{subfigure}

\vspace{0.5em}
\begin{subfigure}[t]{0.48\linewidth}
\centering
\includegraphics[width=\linewidth]{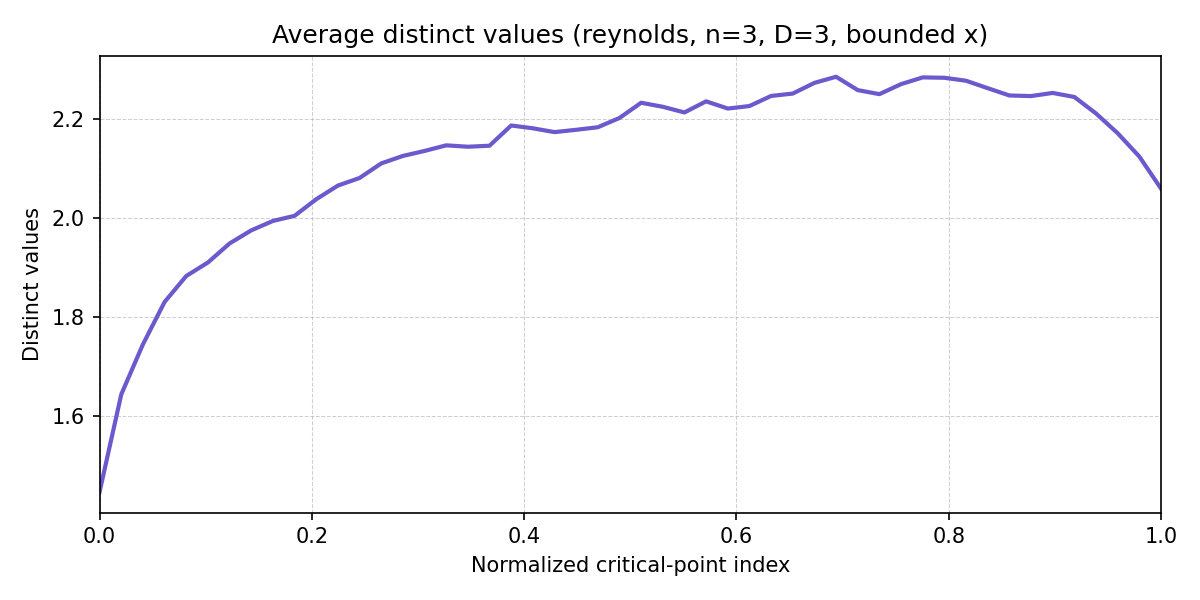}
\caption{Reynolds, $n=3$, $d=3$ (bounded, $+\|x\|^{4}$).}
\end{subfigure}
\hfill
\begin{subfigure}[t]{0.48\linewidth}
\centering
\includegraphics[width=\linewidth]{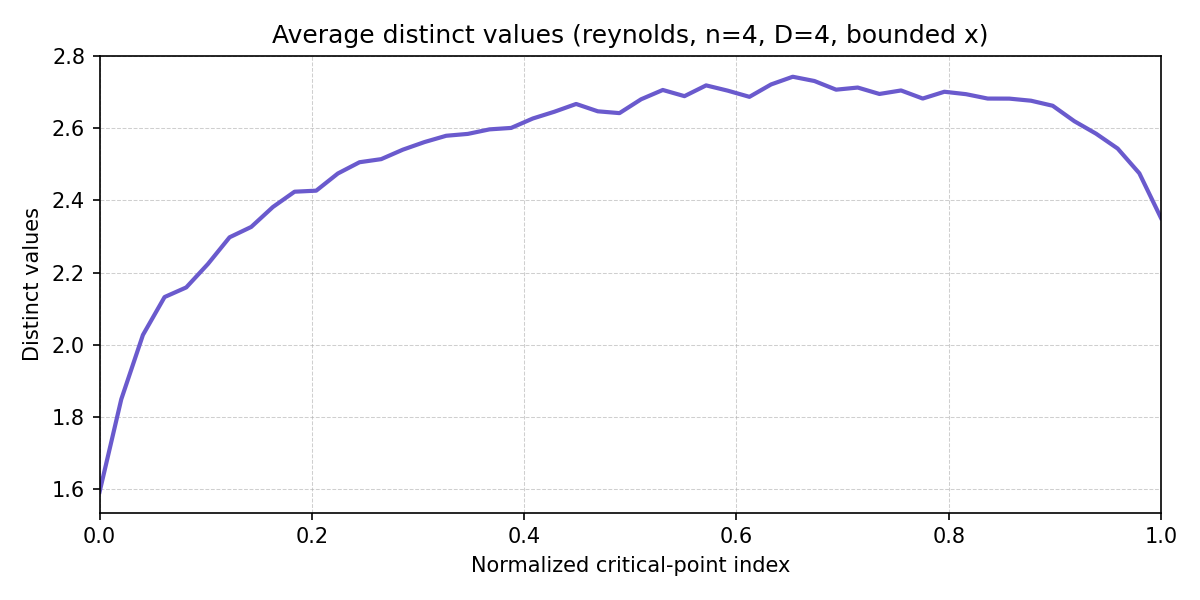}
\caption{Reynolds, $n=4$, $d=4$ (bounded, $+\|x\|^{6}$).}
\end{subfigure}

\vspace{0.5em}
\hfill
\begin{subfigure}[t]{0.48\linewidth}
\centering
\includegraphics[width=\linewidth]{distinct_values_average_continuous_reynolds_n_5_D_3_lb4Bounded_c1.png}
\caption{Reynolds, $n=5$, $d=3$ (bounded, $+\|x\|^{4}$).}
\end{subfigure}
\hfill
\caption{Robust boundary preference when invariance is imposed via Reynolds symmetrization across multiple $(n,d)$ and coercive strengths. In each panel, energetic extremes exhibit fewer distinct coordinate values than the bulk, indicating higher symmetry near the boundary strata.}
\label{fig:boundary_appendix_reynolds}
\end{figure}

\begin{figure}[p]
\centering
\begin{subfigure}[t]{0.48\linewidth}
\centering
\includegraphics[width=\linewidth]{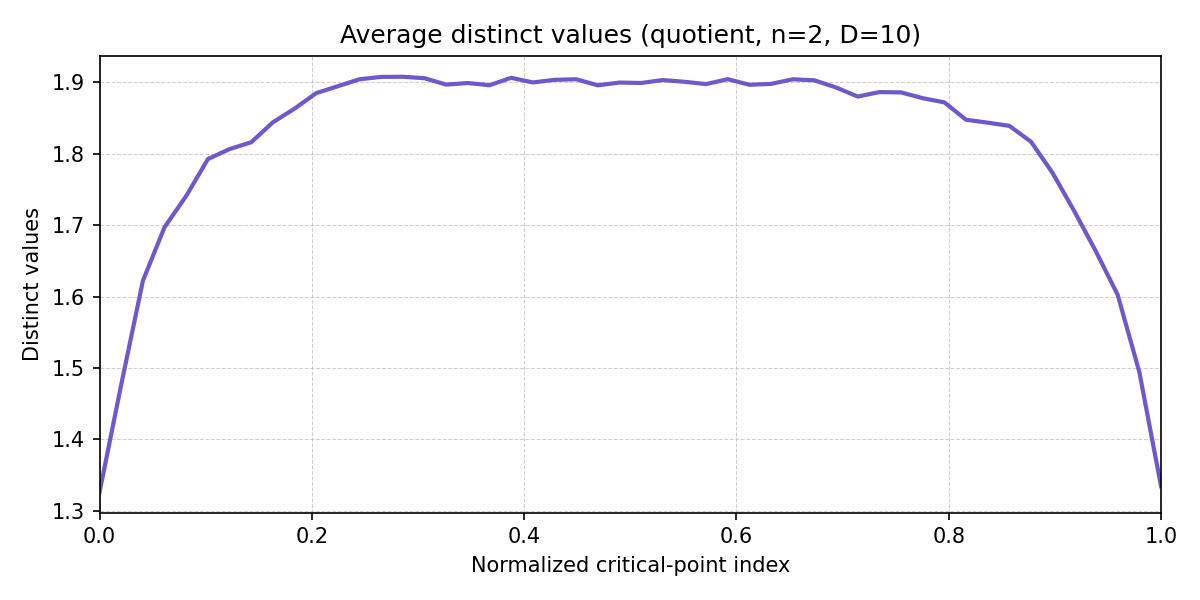}
\caption{Quotient, $n=2$, $d=10$ (unbounded).}
\end{subfigure}
\hfill
\begin{subfigure}[t]{0.48\linewidth}
\centering
\includegraphics[width=\linewidth]{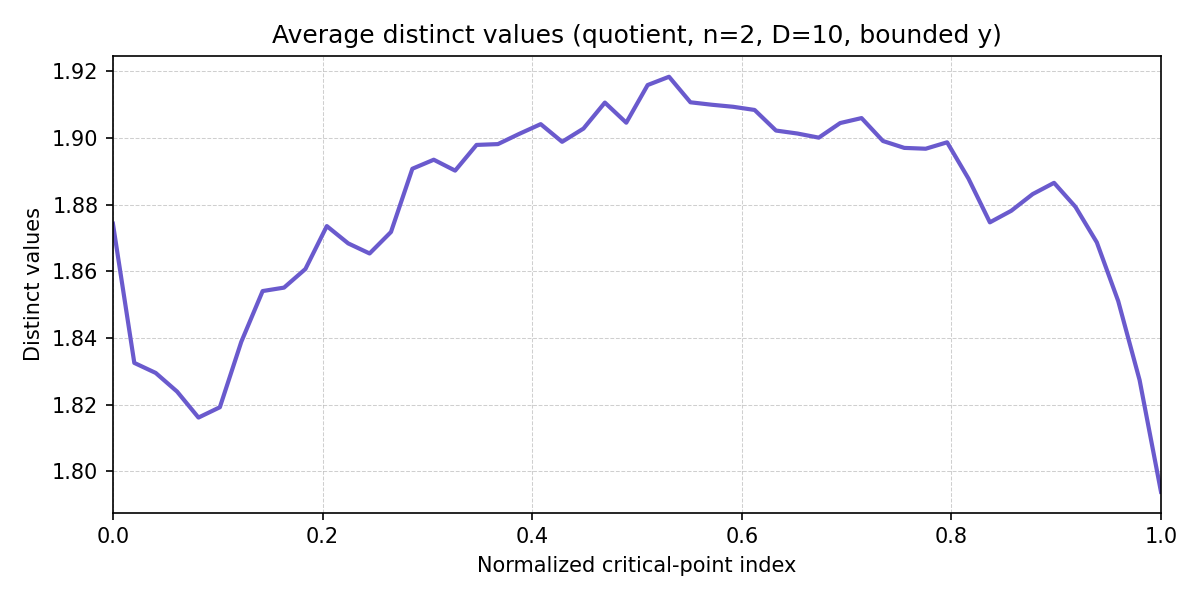}
\caption{Quotient, $n=2$, $d=10$ (bounded on $Y$ via $+\|y\|^{12}$).}
\label{fig:exception_quotient_n2}
\end{subfigure}

\vspace{0.5em}
\begin{subfigure}[t]{0.48\linewidth}
\centering
\includegraphics[width=\linewidth]{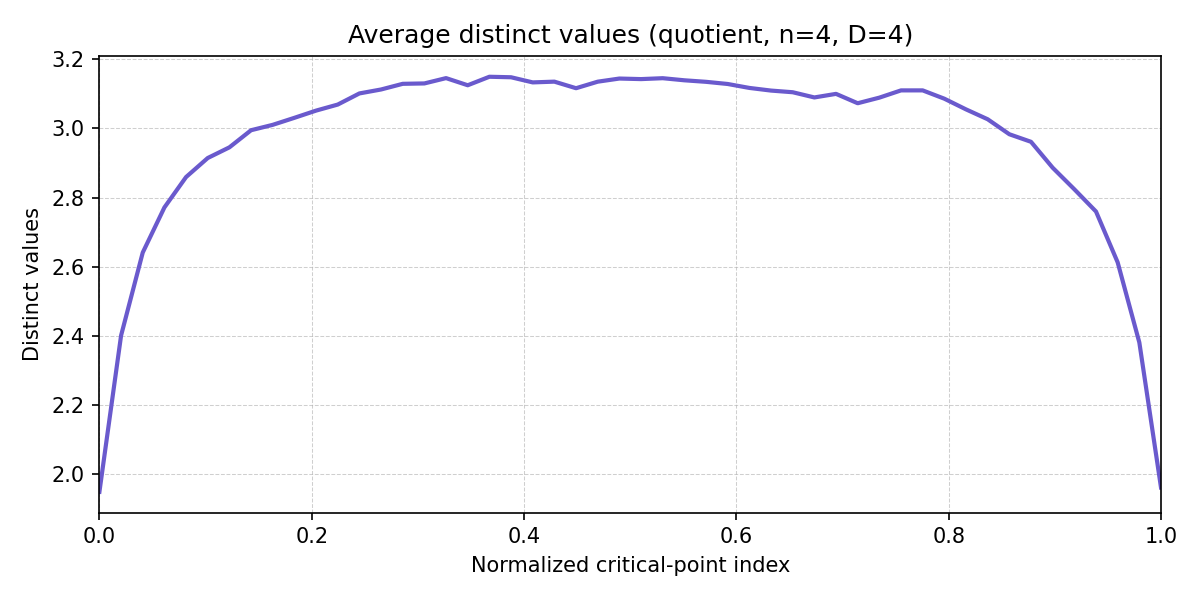}
\caption{Quotient, $n=4$, $d=4$ (unbounded).}
\end{subfigure}
\hfill
\begin{subfigure}[t]{0.48\linewidth}
\centering
\includegraphics[width=\linewidth]{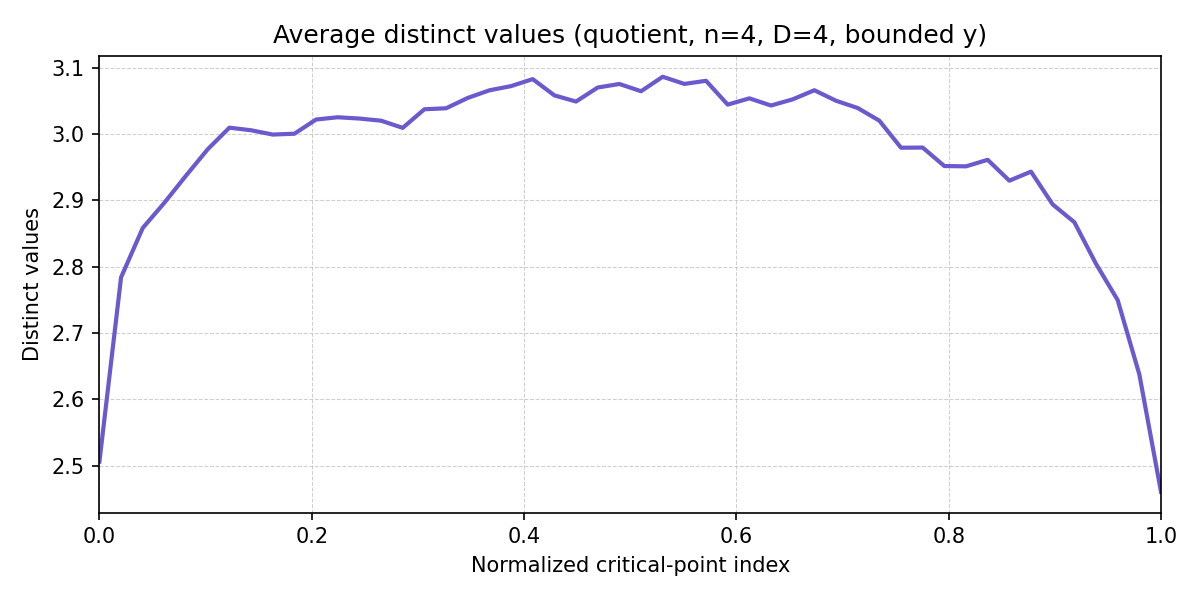}
\caption{Quotient, $n=4$, $d=4$ (bounded on $Y$ via $+\|y\|^{6}$).}
\end{subfigure}
\caption{Boundary preference under the quotient construction $f=P\circ \pi$ across unbounded and $y$-space bounded runs. 
With the notable exception of \ref{fig:exception_quotient_n2}, the $y$-space bounded examples consistently show lower distinct values near the low-energy end (left tail). 
This persistence indicates that the symmetry preference is not generally an artifact of $\tilde f$ being unbounded on $Y(\R)$.}
\label{fig:boundary_appendix_quotient}
\end{figure}

\begin{figure}[p]
\centering
\begin{subfigure}[t]{0.48\linewidth}
\centering
\includegraphics[width=\linewidth]{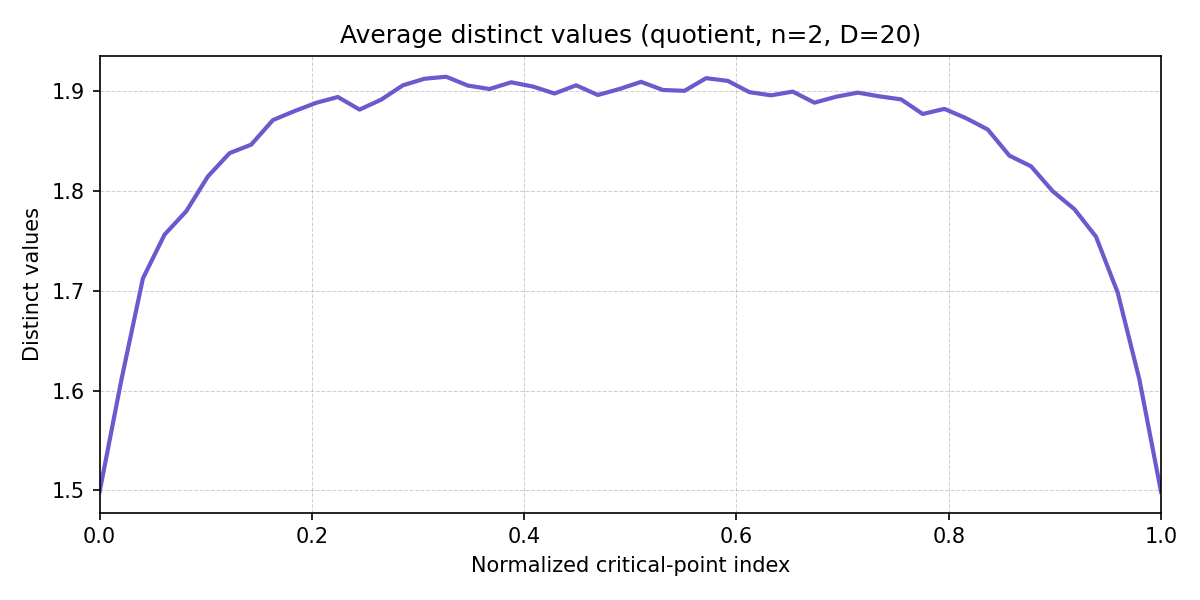}
\caption{Quotient, $n=2$, $d=20$ (unbounded).}
\end{subfigure}
\hfill
\begin{subfigure}[t]{0.48\linewidth}
\centering
\includegraphics[width=\linewidth]{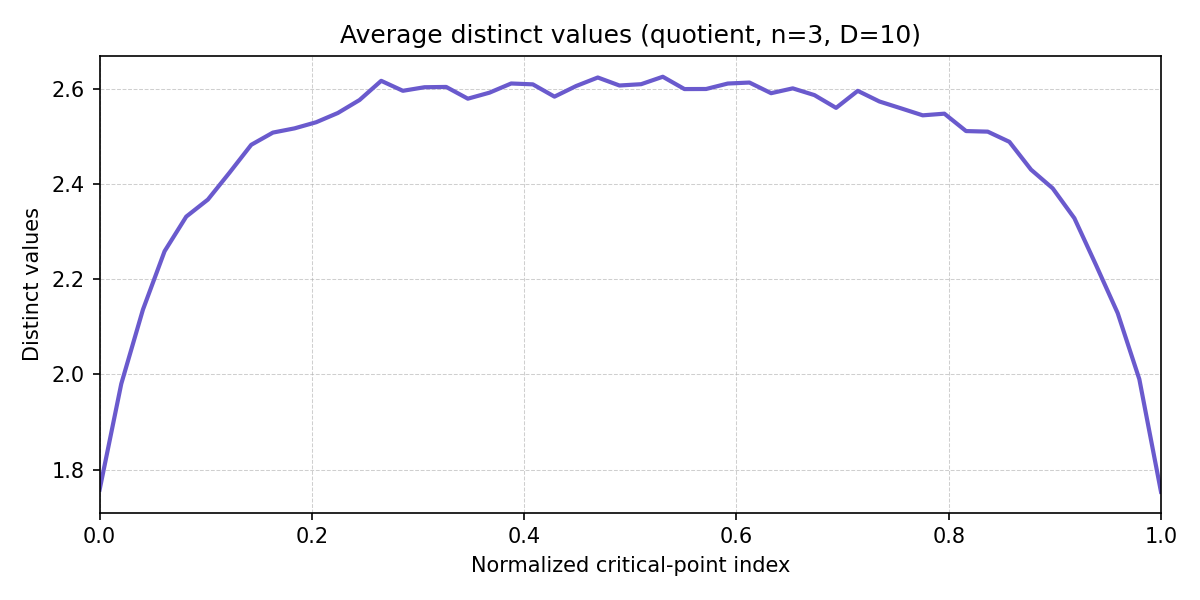}
\caption{Quotient, $n=3$, $d=10$ (unbounded).}
\end{subfigure}
\caption{Additional quotient experiments at higher degree (unbounded cases).}
\label{fig:boundary_appendix_quotient_more}
\end{figure}

\begin{figure}[p]
\centering
\begin{subfigure}[t]{0.48\linewidth}
\centering
\includegraphics[width=\linewidth]{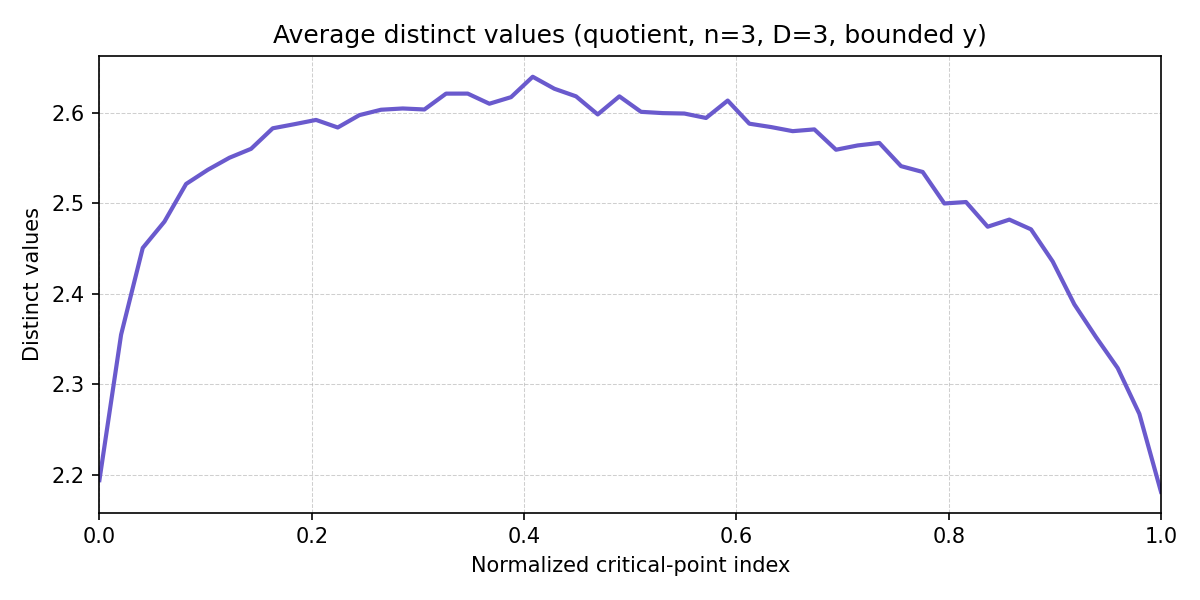}
\caption{Quotient, $n=3$, $d=3$ (bounded on $Y$ via $+\|y\|^{4}$).}
\end{subfigure}
\hfill
\begin{subfigure}[t]{0.48\linewidth}
\centering
\includegraphics[width=\linewidth]{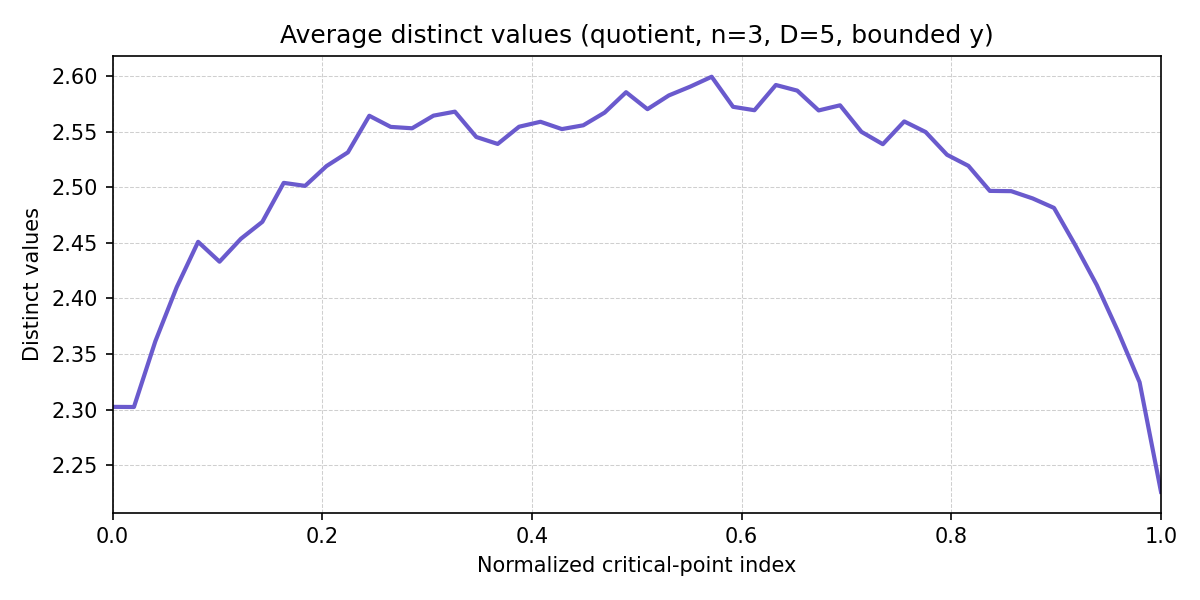}
\caption{Quotient, $n=3$, $d=5$ (bounded on $Y$ via $+\|y\|^{6}$).}
\end{subfigure}

\vspace{0.5em}
\hfill
\begin{subfigure}[t]{0.48\linewidth}
\centering
\includegraphics[width=\linewidth]{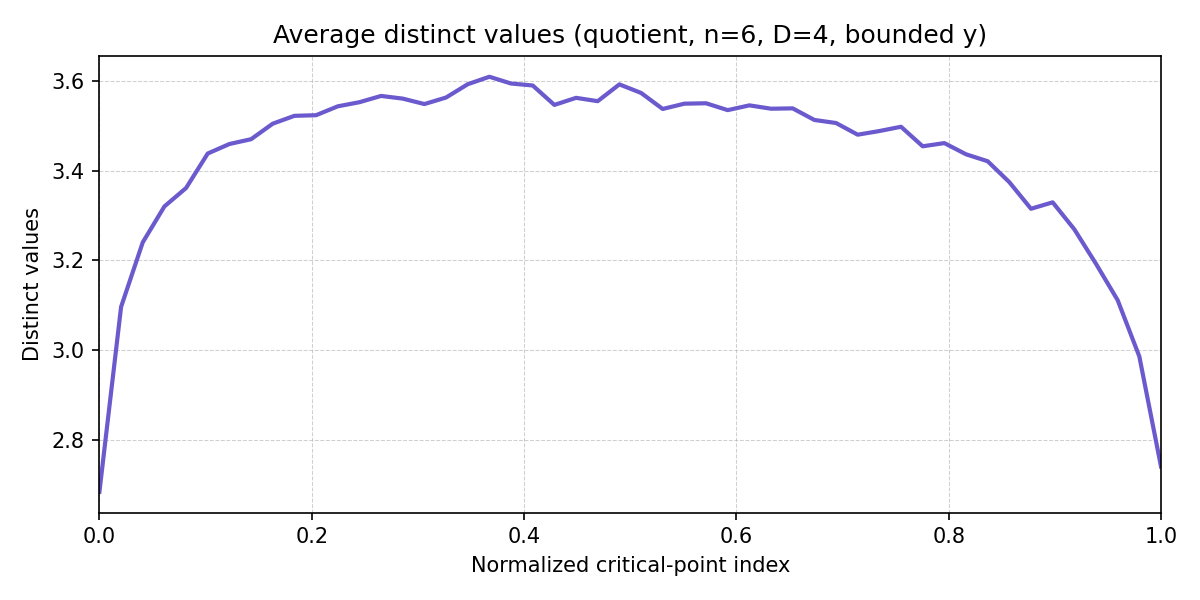}
\caption{Quotient, $n=6$, $d=4$ (bounded on $Y$ via $+\|y\|^{6}$).}
\end{subfigure}
\hfill
\caption{Further $y$-space bounded quotient runs across additional $(n,d)$.}
\label{fig:boundary_appendix_quotient_yspace}
\end{figure}
\begin{figure}[p]
\centering
\begin{subfigure}[t]{0.48\linewidth}
\centering
\includegraphics[width=\linewidth]{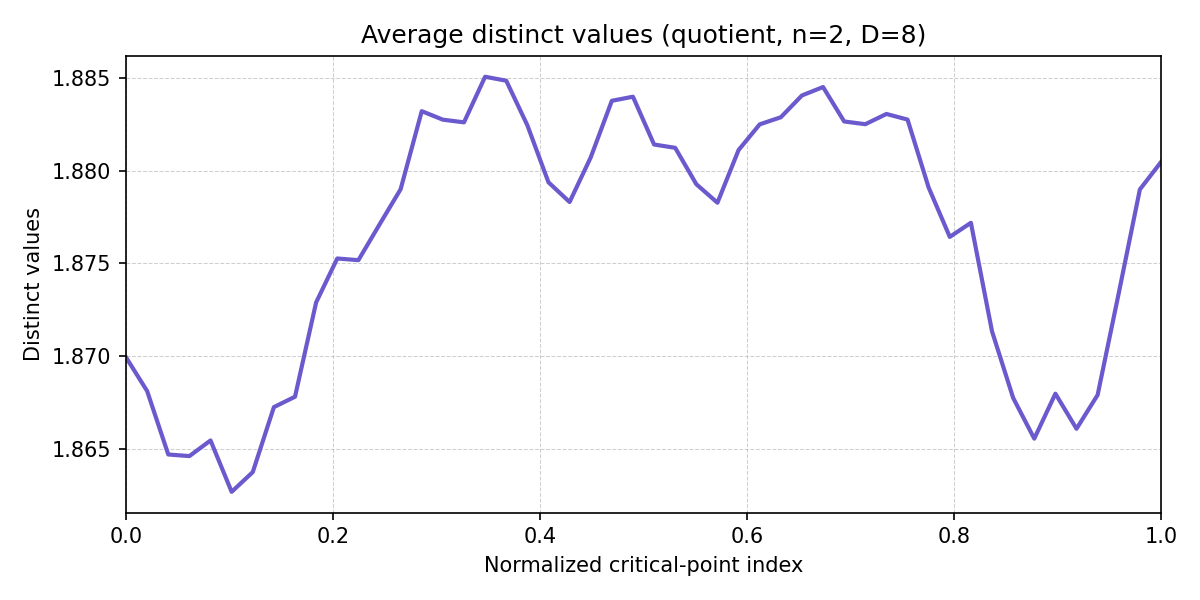}
\caption{$n=2$, $d=8$ (circle in $Y(\R)$).}
\end{subfigure}
\hfill
\begin{subfigure}[t]{0.48\linewidth}
\centering
\includegraphics[width=\linewidth]{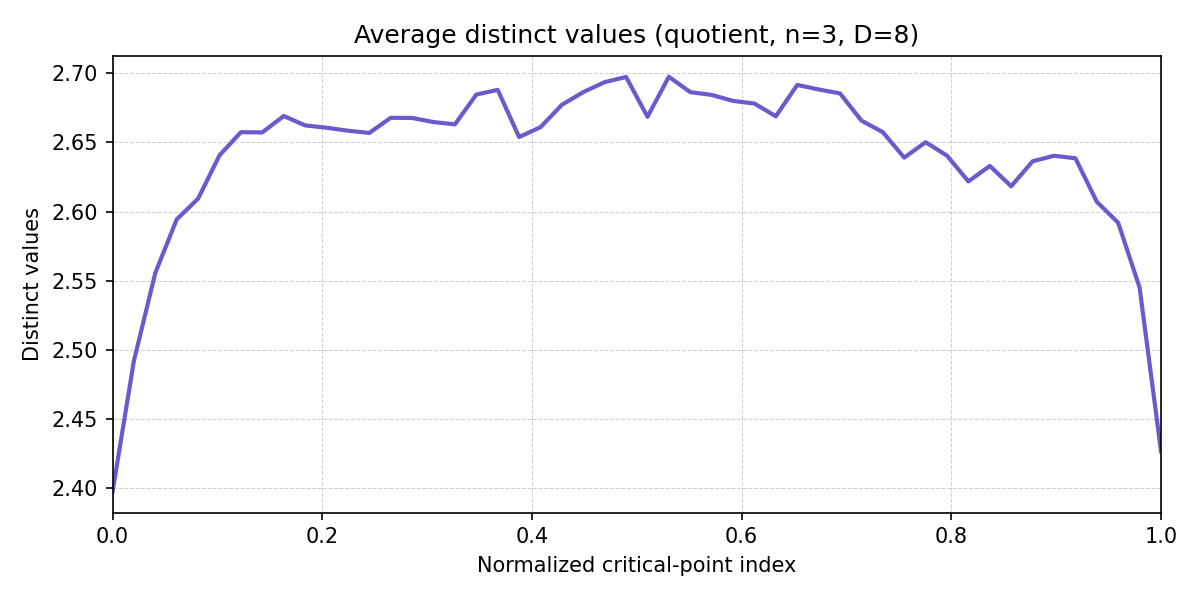}
\caption{$n=3$, $d=8$ (sphere in $Y(\R)$).}
\end{subfigure}
\caption{Quotient construction $f=P\circ\pi$ under ES-sphere normalization $\|e(x)\|^2=\sum_{k=1}^n e_k(x)^2=1$ in $Y(\R)$. For $n=2$ the distinct-values profile is nearly flat, while for $n=3$ the low-energy end still exhibits fewer distinct values (higher symmetry).}
\label{fig:boundary_appendix_quotient_sphere}
\end{figure}
\FloatBarrier
\section{Disentangling the Driver of the Active Constraint}
\label{app:mechanism_disentanglement}

We have proposed the "Active Constraint" mechanism to explain the accumulation of minima at the boundary $\partial L$. This mechanism implies the existence of a "global gradient" that drives optimization trajectories outward.
The fundamental question remains: \textbf{What is the origin of this gradient?}

One hypothesis, translating the statistical arguments of Wales \cite{wales1997global} into our geometric framework, is that this gradient is entropic in nature, driven by the \textit{variance profile} of the ensemble.
Alternatively, the gradient may be a generic topological consequence of optimizing a multimodal function over a restricted domain---much like the ``Carpet on the Himalayas'' analogy---persisting even in the absence of variance fluctuations.

\subsection*{1. The Variance Gradient Hypothesis}
Wales originally argued that symmetric configurations are statistically favored because they exhibit higher energy variance. 
To investigate this, we analyze invariant landscapes generated from the \emph{non-homogeneous KSS ensemble} (Definition~\ref{def:nhkss}), utilizing two distinct construction methods to impose symmetry:

\begin{enumerate}
    \item \textbf{The Quotient-KSS Model (Definition~\ref{def:quotient_kss_model}).} The random potential is sampled on the quotient variables $y=\pi(x)$ and pulled back to configuration space by composition with $\pi$.
    \item \textbf{The Reynolds Model (Definition~\ref{def:reynolds_symmetrization}).} A random polynomial is sampled on configuration space and averaged over the $S_n$-action.
\end{enumerate}

A key feature of the Quotient-KSS model is that its pointwise variance is a strictly radial function of the norm $\|y\|$ in the quotient space $Y(\mathbb{R})$.
Since symmetric points in $X(\mathbb{R})$ map to the boundary $\partial L$—which corresponds to points with larger norms in $Y(\mathbb{R})$—the boundary acts as a locus of maximal variance.
Therefore, one might hypothesize that the "Active Constraint" is driven solely by this statistical profile: the system "drifts" to the boundary simply because that is the direction of increasing fluctuations.

\subsection*{2. The Rigorous Precedent: The Bos-Ware Limit}
This hypothesis is strongly supported in the one-dimensional case by Bos and Ware \cite{Bos2020}. They analyzed the global minimum of random Kac polynomials $p(t)$ on a bounded interval $[-a, a]$. They proved that despite the landscape becoming increasingly multimodal as the degree $n$ rises, the probability that the global minimum occurs strictly on the boundary (the endpoints $\{-a, a\}$) remains uniformly bounded away from zero:
\[
\Prob(x_{\min}\in\{-a,a\})\ge c > 0.
\]

For the ensemble they analyzed, the variance $\sigma^2(x)$ is strictly convex and maximized at these endpoints. Thus, in the 1D limit, the boundary preference aligns perfectly with the direction of maximal variance growth.

\subsection*{3. The Sphere Control Experiment: Testing the Hypothesis}
To determine if this variance gradient is the \textit{sole} driver, or if the boundary preference persists generically due to the domain's shape, we performed a control experiment by varying the constraint manifold in the configuration space $X(\mathbb{R})$.

\paragraph{Proof of Intrinsic Geometric Attraction ($n=3$).}
First, we demonstrate that the preference for symmetry exists even when the statistical variance is completely neutralized.
We consider an $S_3$-invariant random model on $X(\mathbb{R})=\mathbb{R}^3$. To neutralize the variance, we restrict the optimization to the level set of the \textit{quotient norm}:
\[
\mathcal{M}_{Y\text{-sphere}} = \{ x \in \mathbb{R}^3 \mid \|\pi(x)\|^2 = 1 \}.
\]
Note that while defined via $\pi$, this is a manifold in the configuration space $X(\mathbb{R})$.
Since in the quotient-KSS model $\operatorname{Var}(\tilde{f}(y))$ depends only on the radial distance $\|y\|$ in the quotient, restricting $x$ to this manifold ensures that:
\[
\operatorname{Var}(f(x)) = \operatorname{Var}(\tilde{f}(\pi(x))) = \text{const}, \quad \forall x \in \mathcal{M}_{Y\text{-sphere}}.
\]
On this manifold, symmetric and asymmetric configurations are statistically indistinguishable in terms of variance.

\textbf{Result:} As shown in Figure \ref{fig:sphere_n3}, despite the constant variance, the global minima exhibit a persistent preference for the $S_3$ symmetric stratum.
\textbf{Conclusion:} For $n=3$, the "Active Constraint" is not driven by variance. The intrinsic, intricate geometry of the boundary is sufficient to attract the global minimum.

\begin{figure}[H]
\centering
\includegraphics[width=0.6\linewidth]{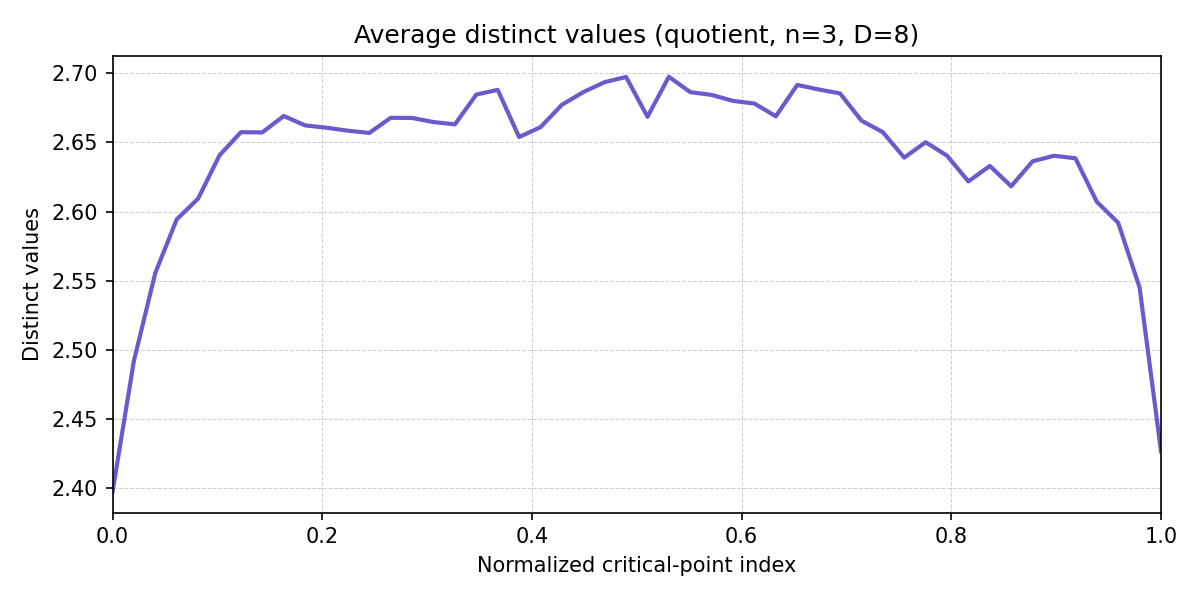}
\caption{\textbf{Geometry acts alone ($n=3$).} Quotient-KSS model (non-homogeneous, $D=8$), restricted to $\{x \in \mathbb{R}^3 \mid \|\pi(x)\|=1\}$. Here, the variance is constant across the domain. The persistence of the symmetry bias proves that the singular geometry of the $n=3$ boundary attracts minima even without a variance gradient.}
\label{fig:sphere_n3}
\end{figure}

\paragraph{The Role of Variance as an Amplifier ($n=2$).}
While geometry alone suffices for $n=3$, the $n=2$ case reveals that variance plays a critical role in \textit{amplifying} this effect for smoother boundaries. We compare optimization on two distinct manifolds in $X(\mathbb{R}) = \mathbb{R}^2$:

\begin{enumerate}
    \item \textbf{The Quotient-Sphere Manifold:}
    \[ \mathcal{M}_{Y\text{-sphere}} = \{ x \in \mathbb{R}^2 \mid \|\pi(x)\|^2 = 1 \}. \]
    As established above, the variance is constant on this set.
    \item \textbf{The Configuration-Sphere Manifold:}
    \[ \mathcal{M}_{X\text{-sphere}} = \{ x \in \mathbb{R}^2 \mid \|x\|^2 = 1 \}. \]
    Moving along this standard circle in configuration space induces radial movement in the quotient space.
    Specifically, the quotient norm $\|\pi(x)\|$ varies along $\mathcal{M}_{X\text{-sphere}}$ and attains its \textbf{maximum} exactly at the symmetric configurations (where $x_1=x_2$).
    Thus, restricting to the standard sphere $\mathcal{M}_{X\text{-sphere}}$ reintroduces a variance gradient that peaks at the boundary.
\end{enumerate}

This geometric distinction is visualized in Figure \ref{fig:n2_sphere_geometry}.

\begin{figure}[H]
\centering
\includegraphics[width=0.7\linewidth]{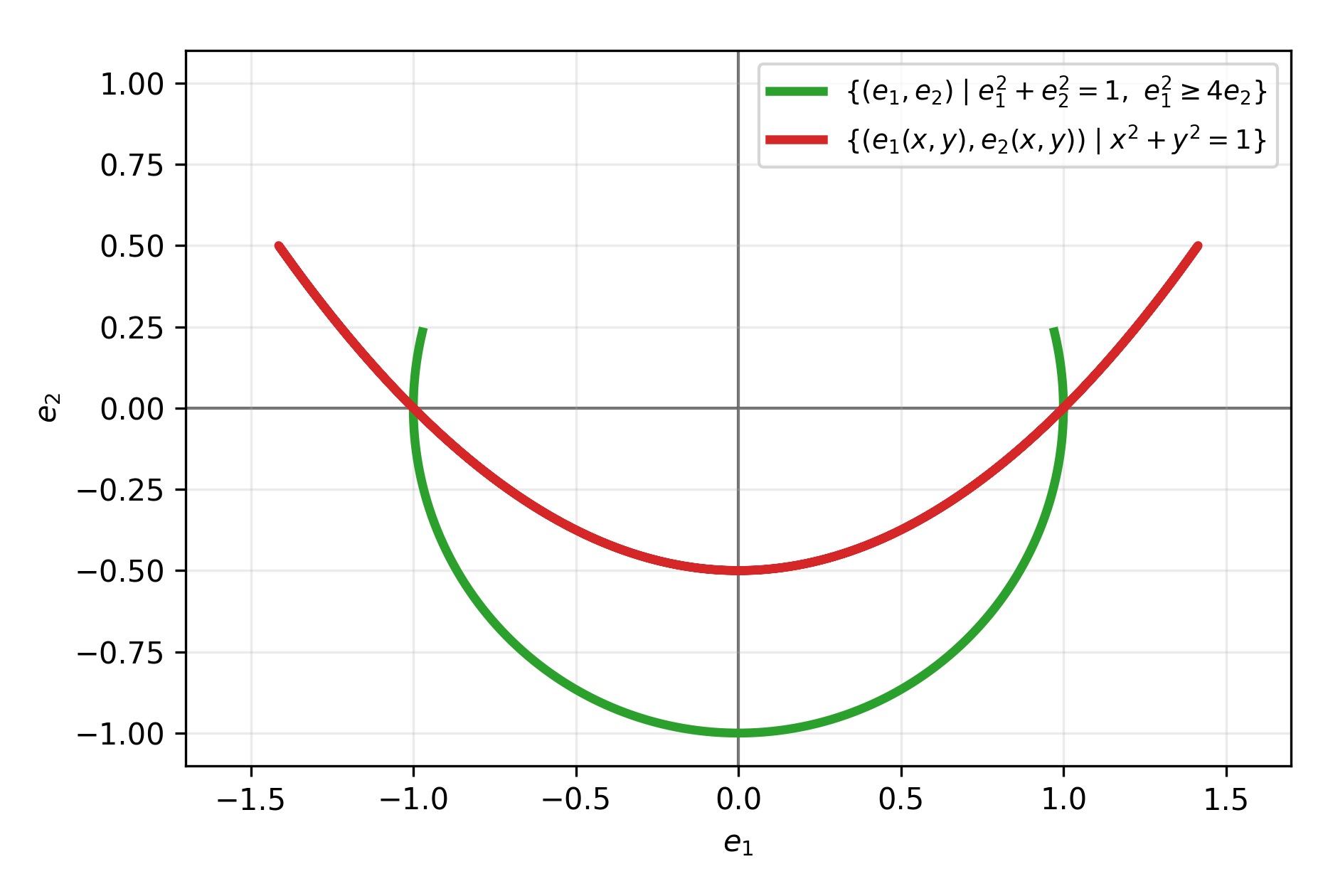}
\caption{\textbf{Two constraints viewed in $Y(\mathbb{R})$.}
\textbf{Setup:} $y=\pi(x)$ and the random potential is the quotient-KSS model, whose variance depends only on $\|y\|$.
\textbf{Green:} The image of $\mathcal{M}_{Y\text{-sphere}}$ (where $\|\pi(x)\|=1$). It lies on a circle in $Y$, so variance is constant.
\textbf{Red:} The image of $\mathcal{M}_{X\text{-sphere}}$ (where $\|x\|=1$). This curve extends radially outward, reaching maximal $\|y\|$ (and thus maximal variance) at the symmetric boundary points.
This illustrates how standard physical constraints ($\|x\|=1$) naturally couple symmetry with high variance.}
\label{fig:n2_sphere_geometry}
\end{figure}

\textbf{Results:} The impact of this difference is drastic (Figure \ref{fig:n2_comparison_panels}):
\begin{itemize}
    \item On $\mathcal{M}_{Y\text{-sphere}}$ (Green, constant variance), the boundary preference vanishes (Fig \ref{fig:sphere_n2_flat}). The smooth fold of the $n=2$ boundary is too weak to trap minima on its own.
    \item On $\mathcal{M}_{X\text{-sphere}}$ (Red, variable variance), the boundary preference returns (Fig \ref{fig:sphere_n2_bias}). The system "drifts" to the symmetric boundary because that is where the variance is maximized.
\end{itemize}

\begin{figure}[H]
\centering
\begin{subfigure}[t]{0.32\linewidth}
\centering
\includegraphics[width=\linewidth]{distinct_values_average_continuous_quotient_n_2_D_8_esSphere.png}
\caption{Quotient-KSS (non-homogeneous, $D=8$)\newline ($\|\pi(x)\|=1$) $\to$ \textbf{No Bias}}
\label{fig:sphere_n2_flat}
\end{subfigure}
\hfill
\begin{subfigure}[t]{0.32\linewidth}
\centering
\includegraphics[width=\linewidth]{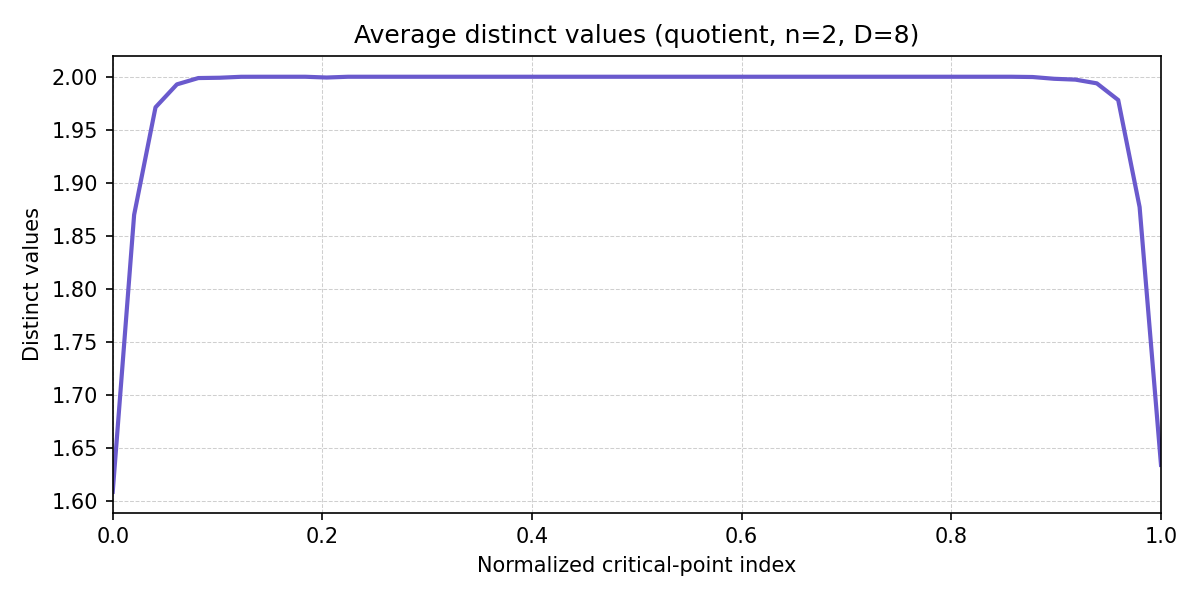}
\caption{Quotient-KSS (homogeneous, $D=8$)\newline ($\|x\|=1$) $\to$ \textbf{Bias Returns}}
\label{fig:sphere_n2_bias}
\end{subfigure}
\hfill
\begin{subfigure}[t]{0.32\linewidth}
\centering
\includegraphics[width=\linewidth]{distinct_values_average_continuous_reynolds_n_2_D_10_lb12Bounded_c1.png}
\caption{Reynolds-symmetrized KSS ($D=10$) with $+\|x\|^{12}$\newline (Unbounded) $\to$ \textbf{Bias}}
\label{fig:reynolds_n2}
\end{subfigure}
\caption{\textbf{Variance as an Amplifier ($n=2$).} (a) Optimizing on $\mathcal{M}_{Y\text{-sphere}}$ (constant variance) shows no boundary accumulation. (b) Optimizing on $\mathcal{M}_{X\text{-sphere}}$ reintroduces the variance gradient, restoring the bias. (c) The full unbounded model confirms the robust effect of radial drift.}
\label{fig:n2_comparison_panels}
\end{figure}

\subsection*{Conclusion: Variance as an Amplifier}
We conclude that the "Active Constraint" operates through a hierarchy of drivers:
\begin{enumerate}
    \item \textbf{Intrinsic Driver:} For complex domains (as in $n=3$), the boundedness and shape of the real image alone are sufficient to push the global minimum to the boundary. The multimodal landscape naturally attains its minimum at the boundary of the domain.
    \item \textbf{Statistical Amplifier:} For smoother domains (as in $n=2$), the intrinsic effect is weak. In these cases, the \textbf{variance gradient} acts as a critical amplifier. By increasing fluctuations at the boundary, it reinforces the drift and ensures the global minimum settles on the symmetric locus.
\end{enumerate}
Thus, we assume that the "Wales mechanism" (high variance at symmetry) acts as a powerful catalyst that amplifies the generic geometric tendency.

\end{document}